\documentclass[12pt]{amsart}

\usepackage{amsmath}
\usepackage{tikz,fp}
\usetikzlibrary{calc}
\usetikzlibrary{positioning}
\usetikzlibrary{decorations.markings,decorations.pathmorphing}
\usetikzlibrary{arrows.meta}
\usepackage{graphicx,color}
\usepackage{float}
\usepackage{amssymb}
\usepackage{pgfplots}
\usepackage{xcolor}
\usepackage{mathtools}
\usepackage{enumerate}
\usepackage{imakeidx}
\usepackage{amsthm}
\usepackage{tabulary}
\usepackage{booktabs}
\usepackage{setspace}

\newtheorem{theorem}{Theorem}

\newtheorem{definition}[theorem]{Definition}

\newtheorem{lemma}[theorem]{Lemma}

\newtheorem{proposition}[theorem]{Proposition}

\newtheorem{remark1}{remark}
\newtheorem{remark}[remark1]{Remark}

\makeatletter
\def\calcLength(#1,#2)#3{
\pgfpointdiff{\pgfpointanchor{#1}{center}}
             {\pgfpointanchor{#2}{center}}
\pgf@xa=\pgf@x
\pgf@ya=\pgf@y
\FPeval\@temp@a{\pgfmath@tonumber{\pgf@xa}}
\FPeval\@temp@b{\pgfmath@tonumber{\pgf@ya}}
\FPeval\@temp@sum{(\@temp@a*\@temp@a+\@temp@b*\@temp@b)}
\FProot{\FPMathLen}{\@temp@sum}{2}
\FPround\FPMathLen\FPMathLen5\relax
\global\expandafter\edef\csname #3\endcsname{\FPMathLen}
}
\makeatother

\newcommand{\udots}{\mathinner{\mskip1mu\raise1pt\vbox{\kern7pt\hbox{.}}
  \mskip2mu\raise4pt\hbox{.}\mskip2mu\raise7pt\hbox{.}\mskip1mu}}
\newcommand{\real}{\mathbb{R}}

\newcommand{\SO}{\textrm{\normalfont{SO}}_3(\mathbb{R})}

\newcommand{\bsm}[1]{\boldsymbol{#1}}

\DeclareMathOperator{\arccot}{arccot}

\usepackage[utf8]{inputenc}
 
\usepackage[english]{babel}
 
\author{Cong Zhou}
\makeindex

\usepackage{pgfplots}
\usepackage{xxcolor}

\usetikzlibrary{decorations.pathreplacing}

 \tikzset{
  jumpdot/.style={mark=*,solid},
  exclr/.append style={jumpdot,color=red,fill=white},
  exclb/.append style={jumpdot,color=blue,fill=white},
  incl/.append style={jumpdot},
}

\pgfdeclareradialshading[tikz@ball]{ball}{\pgfqpoint{0bp}{0bp}}{%
 color(0bp)=(tikz@ball!0!white);
 color(5bp)=(tikz@ball!97!black);
 color(10bp)=(tikz@ball!90!black);
 color(15bp)=(tikz@ball!70!black);
 color(20bp)=(black!70);
 color(30bp)=(black!70)}
\makeatother

\tikzset{viewport/.style 2 args={
    x={({cos(-#1)*1cm},{sin(-#1)*sin(#2)*1cm})},
    y={({-sin(-#1)*1cm},{cos(-#1)*sin(#2)*1cm})},
    z={(0,{cos(#2)*1cm})}
}}

\pgfplotsset{only foreground/.style={
    restrict expr to domain={rawx*\CameraX + rawy*\CameraY + rawz*\CameraZ}{-0.05:100},
}}
\pgfplotsset{only background/.style={
    restrict expr to domain={rawx*\CameraX + rawy*\CameraY + rawz*\CameraZ}{-100:0.05}
}}

\def\addFGBGplot[#1]#2;{
    \addplot3[#1,only background, opacity=0.25] #2;
    \addplot3[#1,only foreground] #2;
}

\newcommand{\bigslant}[2]{{\raisebox{.2em}{$#1$}\left/\raisebox{-.2em}{$#2$}\right.}}

\newcommand{\vt}{\boldsymbol{t}}

\usepackage[perpage]{footmisc}

\usepackage[a4paper, total={6in,8.6in}, centering]{geometry}

\DeclareMathOperator*{\essinf}{ess\,inf}

\usepackage{bm}
\renewcommand{\boldsymbol}[1]{\bm{#1}}

\newcommand{\mQ}{{\boldsymbol{Q}}}
\newcommand{\mP}{{\boldsymbol{P}}}

\newcommand{\mat}[1]{\boldsymbol{#1}}
\newcommand{\vet}[1]{\boldsymbol{#1}}

\DeclareMathOperator*{\esssup}{ess\,sup}

\usepackage{caption}
\usepackage{subcaption}

\title[The Space of Curves in the Sphere]{On the space of  $C^1$ regular curves on sphere with constrained curvature }

\address{Cong Zhou}
\email{congzhou90@gmail.com}

\makeindex

\begin{document}

\begin{abstract}  Let   $\mathcal{P}_{\kappa_1}^{\kappa_2}(\mP, \mQ)$ denote the set of  $C^1$ regular curves  in the $2$-sphere $\mathbb{S}^2$  that start and end at given points with the corresponding Frenet frames  $\mat{P}$ and $\mat{Q}$, whose tangent vectors  are Lipschitz continuous, and their a.e. existing geodesic curvatures have essentially bounds  in $(\kappa_1, \kappa_2)$, $-\infty<\kappa_1<\kappa_2<\infty$.  In this article, firstly we study the geometric property of the curves in $\mathcal{P}_{\kappa_1}^{\kappa_2}(\mP, \mQ)$. We introduce the concepts of the  lower and upper curvatures at any point of a $C^1$ regular curve and prove that a $C^1$ regular curve is in  $\mathcal{P}_{\kappa_1}^{\kappa_2}(\mP, \mQ)$  if and only if  the infimum of  its lower curvature and the supremum of its  upper curvature are constrained in $(\kappa_1,\kappa_2)$. Secondly we prove that  the $C^0$ and $C^1$ topologies  on $\mathcal{P}_{\kappa_1}^{\kappa_2}(\mP, \mQ)$  are the same.
Further,  we show that a curve in   $\mathcal{P}_{\kappa_1}^{\kappa_2}(\mP, \mQ)$ can be determined by the solutions of differential equation $\Phi'(t) = \Phi(t)\Lambda(t)$ with $\Phi(t)\in\SO$ with special constraints to $\Lambda(t)\in\mathfrak{so}_3(\mathbb{R})$ and give a complete metric on  $\mathcal{P}_{\kappa_1}^{\kappa_2}(\mP, \mQ)$ such that it becomes a (trivial) Banach manifold.

\end{abstract}

\maketitle

\tableofcontents

\section{Introduction}\label{introduction}

In this article, we study the geometric and topological properties of the space of $C^1$ regular curves on the unit $2$-sphere $\mathbb{S}^2$ with the ``curvatures'' constrained in an interval. Let  $\gamma(t), t\in [0,1]$ denote a $C^1$ regular curve in $\mathbb{S}^2$ whose tangent vector $\dot{\gamma}$  is Lipschitz continuous.  This implies that $\ddot{\gamma}(t)$ exists for a.e. $t$.  
Reparameterizing $\gamma$ with arc-length $s$,  $\gamma'(s)$ is also Lipschitz continuous for $s$ and  $\gamma''(s)={\vet{t}}'(s)$ exists for a.e. $s$. Moreover,
$${\vet{t}}'(s)=-\gamma(s)+\kappa(s){\vet{n}}(s), \quad a.e. \quad s,$$
where ${\vet{t}}(s)$ and ${\vet{n}}(s)$ are the unit tangent vector and unit normal vector at $\gamma(s)$, and $\kappa(s)$ is called geodesic curvature at $\gamma(s)$ or $\gamma(t)$.

Consider   the set of all $C^1$ regular curves  in $\mathbb{S}^2$  that start and end at given points with given directions.  Precisely, denote by $\mathcal{I}(\mP,\mQ)$  the set of all $C^1$ regular curves $\gamma$  in $\mathbb{S}^2$  with  Frenet frames $\mathfrak{F}_\gamma(0)=\mP\in \SO$ and $\mathfrak{F}_\gamma(1)=\mQ\in \SO$.

We study the subset  $\mathcal{P}_{\kappa_1}^{\kappa_2}(\mP, \mQ)$ consisting of the curves  in $\mathcal{I}(\mP,\mQ)$
whose tangent vector  is Lipschitz continuous and whose geodesic curvature $\kappa(t)$ (which exists for  a.e. $t$) satisfies 
\begin{equation*}\displaystyle
\kappa_1<\essinf_{t\in[0,1]} \kappa(t)\leq \esssup_{t\in[0,1]} \kappa(t)<\kappa_2,
\end{equation*}
where $\kappa_1<\kappa_2$ are real numbers, $\essinf \kappa(t)$ and $\esssup \kappa(t)$ denote the essential infimum and essential supremum of $\kappa(t)$, respectively. 


Our study of  the geometry and  topology of  $\mathcal{P}_{k_1}^{k_2}(\mP, \mQ)$   is motivated by  the investigation on the topologies of $C^r$ regular curves in $\mathbb{S}^2$, $r\geq 1$. Here we briefly recall  some results in this topic.
In 1956, Smale \cite{smale} proved that the space of $C^r$ ($r\geq 1$) regular closed curves on $\mathbb{S}^2$, has only two connected components. Each of them are homotopically equivalent to $\SO\times\Omega\mathbb{S}^3$, where $\Omega\mathbb{S}^3$ denotes the space of all continuous closed curves in $\mathbb{S}^3$ with the $C^0$ topology.  Later in 1970, Little \cite{little} proved 
that  there are a total of $6$ second order non-generate regular homotopy classes of  $C^r$, $r\geq 2, $ regular closed curves in $\mathbb{S}^2$. 
In 1999, Shapiro and Khesin \cite{shakhe} began to study the topology of the space of all smooth regular locally convex curves (not necessarily closed) in $\mathbb{S}^2$ which start and end at given points with given directions. They showed that
the space of such curves consists of $3$ connected components if there exists a disconjugate curve connecting them. Otherwise the space consists of $2$ connected components. 
During 2009-2012, in \cite{sald1}, \cite{sald2} and \cite{sald},  Saldanha did several further works on the higher homotopy properties of the space of locally convex curves on $\mathbb{S}^2$ and gave an explicit homotopy for space of locally convex curves with prescribed initial and final Frenet frames.
Recently, in 2013, Saldanha and Zühlke \cite{salzuh} extended Little's result to the space of $C^r, r\geq 2,$ regular closed curves with geodesic curvature constrained in an open interval  $-\infty \leq \kappa_1 < \kappa_2 \leq +\infty$.
Moreover, they conjectured the $(n-1)$-th and $n$-th connected components $\mathcal{L}_{n-1}$ and $\mathcal{L}_{n}$ in a  theorem of them  (\cite{salzuh}, Theorem B) to be homotopically equivalent to $(\Omega\mathbb{S}^3) \vee \mathbb{S}^{n_1}\vee \mathbb{S}^{n_2}\vee\mathbb{S}^{n_3}\vee \cdots$, where $n$ depends on $\kappa_1$ and $\kappa_2$. In \cite{zhou}, we considered the subspace $\mathcal{L}_{\kappa_1}^{\kappa_2}(P,Q)$ of the space $\mathcal{I}_{\kappa_1}^{\kappa_2}(P,Q)$ (see its  definition in Section \ref{notation}) and proved  the existence of a non-trivial map $F:\mathbb{S}^{n_1}\to\mathcal{L}_{\kappa_1}^{\kappa_2}(P,Q)$, where  the dimension $n_1$ are linked to the maximum number of arcs of angle $\pi$ for each of four types of ``maximal'' critical curves. This result is consistent with the conjecture.
We refer the readers to the articles \cite{shakhe}, \cite{little},  \cite{sald1}, \cite{sald2}, \cite{sald},  \cite{salsha}, \cite{shasha}, \cite{salzuh}, \cite{smale},  \cite{shap}, and references therein for more knowledge on this subject.

In this article, we obtain some results related to our research in \cite{zhou} and the work by Saldanha and Zühlke in \cite{salzuh}.
First, we discuss the geometric character of  the curves in  $\mathcal{P}_{\kappa_1}^{\kappa_2}(\mP, \mQ)$.
In order to do this, in Section 2, we introduce the concepts of  upper curvature $\kappa^+_{\gamma}(t)$ and lower curvature $\kappa^-_{\gamma}(t)$ at the point $\gamma(t)$ of a general $C^1$ regular curve $\gamma$  by comparing the curve with the corresponding families of the circles that are tangent to the curve at $\gamma(t)$. Both curvatures restricted to $C^2$ curves  coincide with the usual \emph{geodesic curvature}. We proved  that any  $C^1$ regular curve is in  $\mathcal{P}_{\kappa_1}^{\kappa_2}(\mP, \mQ)$  if and only if it satisfies   $\displaystyle\kappa_1<\inf_{t\in[0,1]} \kappa^-_{\gamma}\leq \sup_{t\in[0,1]} \kappa^+_{\gamma}<\kappa_2$, that is,
\begin{theorem}\label{thm1}$\mathcal{P}_{\kappa_1}^{\kappa_2}(\mat{P},\mat{Q})=\mathcal{S}_{\kappa_1}^{\kappa_2}(\mat{P},\mat{Q})$,
where $\mathcal{S}_{\kappa_1}^{\kappa_2}(\mat{P},\mat{Q})$ denotes the subset of $\mathcal{I}(\mP,\mQ)$ consisting of the curves which satisfy 
 $\displaystyle\kappa_1<\inf_{t\in[0,1]} \kappa^-_{\gamma}\leq \sup_{t\in[0,1]} \kappa^+_{\gamma}<\kappa_2$.
 \end{theorem}

Next, we consider the $C^0$ and $C^1$ topologies of the space $\mathcal{P}_{\kappa_1}^{\kappa_2}(\mP, \mQ)$. Although these two topologies are different for the space $\mathcal{I}(\mP,\mQ)$, they   are the same for $\mathcal{P}_{\kappa_1}^{\kappa_2}(\mP, \mQ)$. We prove that
\begin{theorem} \label{thm-2} The metric spaces $(\mathcal{P}_{\kappa_1}^{\kappa_2}(\mat{P},\mat{Q}),\mat{Q}),d^0)$ and $(\mathcal{P}_{\kappa_1}^{\kappa_2}(\mat{P},\mat{Q}),d^1)$ generate the same topology. Here $d^0$ and $d^1$ are the following metrics
\begin{equation*}d^0(\alpha,\beta) = \max \left\{d\big(\alpha\left(t\right),\beta\left(t\right)\big); t\in [0,1] \right\},
\end{equation*}
 where $d$ denotes the surface distance on $\mathbb{S}^2$. 
\begin{equation*}
d^1(\alpha,\beta) = \max \left\{d\big((\alpha\left(t\right),\dot{\alpha}(t)), (\beta\left(t\right),  \dot{\beta}(t)) \big); t\in [0,1] \right\},
\end{equation*}
where $d$ is the distance measured in the tangent bundle  $\textup{T}\mathbb{S}^2$ with a Riemannian metric induced from $\mathbb{S}^2$.
\end{theorem}

 We remark that for $\mathcal{P}_{\kappa_1}^{\kappa_2}(\mat{P},\mat{Q})$, it is known that the compact-open topology is also equivalent to the $C^0$ topology induced by the metric $d^0$. 
 
Lastly,  we equip $\mathcal{P}_{\kappa_1}^{\kappa_2}(\mP, \mQ)$ with  a special norm  so that it becomes a Banach space, hence   a trivial Banach manifold. The approach is to
write the Frenet frame of a related $C^1$ regular curve in $\mathbb{S}^2$ as a weak solution of a differential equation and use the similar idea  in \cite{salzuh}. 
 We obtain that
\begin{theorem}\label{thm3} $\mathcal{P}_{\kappa_1}^{\kappa_2}(\mat{P},\mat{Q})$ can be furnished a complete norm so that it is  a Banach space, hence a trivial Banach manifold.
\end{theorem}
We remark that the topology of the  Banach manifold $\mathcal{P}_{\kappa_1}^{\kappa_2}(\mat{P},\mat{Q})$ obtained above is not equivalent to $C^0$ topology on $\mathcal{P}_{\kappa_1}^{\kappa_2}(\mat{P},\mat{Q})$ since the latter is not complete.

\section{Definitions and notations}\label{notation}

In this section, we give some definitions and notations. 
Let $\mathbb{S}^2$ denote the unit sphere in the Euclidean space $\mathbb{R}^3$. 
 A $C^1$ regular parameterized curve in $\mathbb{S}^2$ is a $C^1$ map $\gamma:I\to\mathbb{S}^2$ such that the tangent vector $\dot{\gamma}(t)\neq 0$ for all $t\in I$, where $I=[a,b]\subset \mathbb{R}$.
In other words, a $C^1$ regular parameterized curve in  $\mathbb{S}^2$ is a $C^1$ immersion of $I$  into $\mathbb{S}^2$.  Now we recall the definition of the space of $C^1$ regular curves. 


There is an equivalence relation  $\sim$ in  the set of all $C^1$ regular parameterized curves in $\mathbb{S}^2$. Any two $C^1$  regular parameterized curves in $\mathbb{S}^2$
$$ \alpha:I_\alpha\to\mathbb{S}^2 \quad \text{and} \quad \beta:I_\beta\to\mathbb{S}^2 $$
\noindent are called equivalent if they are the same up to a \emph{reparameterization}, that is, there exists a $C^1$ bijection $\bar{t}:I_\alpha\to I_\beta$, satisfying $\frac{d\bar{t}}{dt}>0$ and
$$ \alpha(t) = \left(\beta\circ\bar{t}\right)(t)\quad \forall t\in I_\alpha. $$
  

The space of $C^1$ \emph{regular curves} in $\mathbb{S}^2$ is defined as the quotient space
$$\mathcal{I} = \bigslant{\left\{\gamma: \gamma\text{ is a } C^1 \text{ regular parameterized curve in }\mathbb{S}^2\right\}}{\sim}. $$
\noindent\label{equivpg}By abuse of notation, we will still use $\alpha$ to represent the equivalence class $[\alpha]=\{\beta;\alpha\sim\beta\}\in\mathcal{I}$ and call $\alpha\in\mathcal{I}$ a $C^1$ regular curve in $\mathbb{S}^2$. 



For a $C^1$ regular parameterized curve $\gamma: [0,1]\to\mathbb{S}^2$ with parameter $t$, the \emph{arc-length} $s:[0,1]\to [0,L_\gamma]$ of $\gamma$ is given by
\begin{equation*}
s(t)\coloneqq \int_0^{t} |\dot{\gamma}(t)| dt  ,
\end{equation*}
\noindent where $L_\gamma=\int_0^1|\dot{\gamma}(t)|dt$ is the \emph{length} of $\gamma$. Since $|\dot{\gamma}|> 0$, $s$ is a strictly increasing function. 
Re-parameterizing the curve by arc-length $s$, the curve $\gamma: [0,L_\gamma]\to\mathbb{S}^2$ 
satisfies $|\gamma'(s)|\equiv 1$. 
It is easy to  known that one may reparameterize $\gamma$ proportionally to arc-length so that $\gamma:[0,1]\to\mathbb{S}^2$  has constant speed  $|\dot{\gamma}|\equiv L_{\gamma}$. In this article, unless stated otherwise, a $C^1$ regular curve will be identified with this parameterization.

Throughout this paper, we will use the notation $\vt_\gamma(t)$ to denote the unit tangent vector at $\gamma(t)$, that is, $\vt_\gamma(t)=\gamma'(s)\vert_{s=s(t)}$. Derivatives with respect to $t $ and $s$ will be 
denoted by a $\dot{ }$ and a $'$, respectively.  We use this convention for higher-order derivatives as well.

We may define $C^0$ and $C^1$ metrics in $\mathcal{I}$: Given any two curves $\alpha,\beta:[0,1]\to\mathbb{S}^2$ in $\mathcal{I}$ with constant speeds, 
$$d^0(\alpha,\beta) = \max \left\{d\big(\alpha\left(t\right),\beta\left(t\right)\big); t\in [0,1] \right\},$$
\noindent where $d$ is the surface distance on $\mathbb{S}^2$. 
$$\bar{d}^0(\alpha,\beta) = \max \left\{d\big(\alpha\left(t\right),\beta\left(t\right)\big); t\in [0,1] \right\},$$
where $d$ is the distance measured on $\mathbb{R}^3$.
\begin{equation*}
d^1(\alpha,\beta) = \max \left\{d\big((\alpha\left(t\right),\dot{\alpha}(t)), (\beta\left(t\right),  \dot{\beta}(t)) \big); t\in [0,1] \right\},
\end{equation*}
where $d$ is the distance measured in the tangent bundle  $\textup{T}\mathbb{S}^2$ with a Riemannian metric. This metric is equivalent to the metric
$$\bar{d}^1(\alpha,\beta) = \max \left\{d_1\big(\alpha\left(t\right),\beta\left(t\right)\big)+d_2\big(\dot{\alpha}\left(t\right),\dot{\beta}\left(t\right)\big); t\in [0,1] \right\},$$

\noindent where $d_1$ is the surface distance on $\mathbb{S}^2$ and $d_2$ is the distance on $\mathbb{R}^3$. 

The metrics $d^0$ and $\bar{d}^0$ are  equivalent, and also   the metrics $d^1$ and $\bar{d}^1 $ are  equivalent. 
These metrics above induce corresponding topologies in $\mathcal{I}$. Let $(\mathbb{S}^2)^{[0,1]}$ be  the space of all continuous maps from $[0,1]$ into $\mathbb{S}^2$. Note that $\mathcal{I}\subset (\mathbb{S}^2)^{[0,1]}$. It is well known that the topology induced by the metric $d^0$ is equivalent to the compact-open topology on $(\mathbb{S}^2)^{[0,1]}$ (see Proposition A.13., page 530 in \cite{hatc}).

\bigskip
Now we give the definition of lower and upper curvatures for $C^1$ regular curves. Given a $C^1$ regular curve $\gamma:I\to \mathbb{S}^2$, the \emph{unit normal vector} $\boldsymbol{n}_\gamma$ to $\gamma$ is 
$$\bsm{n}_\gamma(t)=\gamma(t)\times\vt_\gamma(t),$$
where $\times$ denotes the vector product in $\mathbb{R}^3$. If $\gamma$ also has the second derivative $\ddot{\gamma}(t)$ at $\gamma(t)$, the \emph{geodesic curvature} $\kappa_\gamma(s)$ at $\gamma(s)=\gamma(s(t))$ is defined by
\begin{equation}\label{curv}
\kappa_\gamma(s)=\left\langle\vt_\gamma'(s), \bsm{n}_\gamma(s)\right\rangle,
\end{equation} 
\noindent where $s$ is the arc-length of $\gamma$. However, for a $C^1$ regular curve, the geodesic curvature may not be well defined at a point of the curve. Here we establish a weaker definition than the geodesic curvature $\kappa_\gamma(s)$ for $C^1$ regular curves below (see Figure \ref{fig:broader}, for an intuition of this concept).

 Given a $C^1$  regular curve $\gamma:I_1\to\mathbb{S}^2$ and a circle $\zeta:I_2\to\mathbb{S}^2$, we say that $\zeta$ is tangent from \emph{left} to $\gamma$ at $\gamma(t_1)$, with $t_1\in I_1$, if the following conditions are satisfied:

\begin{enumerate}
\item There exists a $t_2\in I_2$ such that $\gamma(t_1)=\zeta(t_2)$ and $\vt_\gamma(t_1)=\vt_\zeta(t_2)$.
\item Denote the center of $\zeta$ by $a$ so that $\zeta$ travels \emph{anti-clockwise} with respect to  $a$. There exists a $\delta>0$ such that:
$$ d(\gamma(t),a)\geq r, \quad \forall t\in (t_1-\delta,t_1+\delta),$$
where  $r$ denotes the radius (measured on sphere) of $\zeta$ in relation to the center $a$ and $d$ is the distance measured on $\mathbb{S}^2$. 
\end{enumerate}
In the same manner, we say that $\zeta$ is tangent from \emph{right} to $\gamma$ at $\gamma(t_1)$ by replacing Condition (2) with:
\begin{enumerate}
\item[(2')] Denote the center of $\zeta$ by $a$ so that $\zeta$ travels \emph{anti-clockwise} with respect to  $a$. There exists a $\delta>0$ such that:
$$ d(\gamma(t),a)\leq r, \quad \forall t\in (t_1-\delta,t_1+\delta),$$ 
where  $r$ denotes the radius (measured on sphere) of the circle $\zeta$ in relation to the center $a$ and $d$ is the distance measured on $\mathbb{S}^2$. 
\end{enumerate}

\begin{definition}\label{definition4}
For a $C^1$ regular curve $\gamma:I\to\mathbb{S}^2$, we define the \emph{upper} and the \emph{lower curvatures}, denoted respectively by $\kappa_\gamma^+$ and $\kappa_\gamma^-$, as follows:
\begin{align*}
\kappa_\gamma^+(t) &=\inf\big\{\cot(r);\text{where $r$ is the radius of a circle tangent from left to $\gamma$ at $\gamma(t)$} \big\},\\
\kappa_\gamma^-(t) &= \sup\big\{\cot(r);\text{where $r$ is the radius of a circle tangent from right to $\gamma$ at $\gamma(t)$} \big\},
\end{align*}
\noindent where $t\in I$. We follow the conventions $\inf\emptyset = +\infty$ and $\sup\emptyset = -\infty$. If $\kappa_\gamma^+(t)\geq\kappa_\gamma^-(t)$ for some $t$, we define the \emph{curvature} of $\gamma$ at $t$ as 
$$\kappa_\gamma(t)\coloneqq\kappa_\gamma^+(t)=\kappa_\gamma^-(t).$$
\end{definition}

Since the radius of a circle tangent from right is greater than or equal to the radius of a circle tangent from left, we have $\kappa_\gamma^+(t)\geq\kappa_\gamma^-(t)$ for all $t\in I$. Refer to Figures \ref{fig:broader} and \ref{fig:ctex} for examples of curves and upper/lower curvatures given by Definition \ref{definition4}.

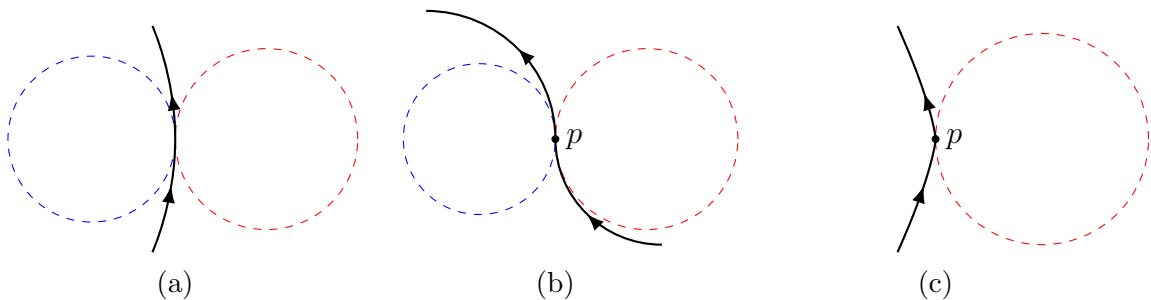
\begin{figure}[b]
\begin{tikzpicture}
\begin{scope}[xshift = -0cm, decoration={markings, mark=at position 0.5 with {\arrowreversed {Latex}}}]
\draw[dashed,blue] (-1.0,0) circle (1.0);
\draw[dashed,red] (1.2,0) circle (1.2);
\draw [postaction={decorate},thick] (0,0) arc (180:270:1.4);
\node[anchor=west] at (0,0) {$p$};
\draw[fill=black] (0,0) circle (.045);
\node[anchor=north] at (0,-1.6) {\small(b)};
\end{scope}
\begin{scope}[xshift = -0cm, decoration={markings, mark=at position 0.5 with {\arrow {Latex}}}]
\draw [postaction={decorate},thick] (0,0) arc (0:90:1.7);
\end{scope}
\begin{scope}[xshift = -5cm, decoration={markings, mark=at position 0.3 with {\arrow {Latex}},mark=at position 0.7 with {\arrow {Latex}}}]
\draw[dashed,blue] (-1.1,0) circle (1.1);
\draw[dashed,red] (1.2,0) circle (1.2);
\draw [postaction={decorate},thick] plot [smooth, tension = 1.0] coordinates {(-.3,-1.5) (0,0) (-.3,1.5)};
\node[anchor=north] at (0,-1.6) {\small(a)};
\end{scope}
\begin{scope}[xshift = 5cm, decoration={markings, mark=at position 0.3 with {\arrow {Latex}}, mark=at position 0.7 with {\arrow {Latex}}}]
\draw[dashed,red] (1.4,0) circle (1.4);
\draw[postaction={decorate},scale=0.5,domain=-1:1,smooth,variable=\y,thick]  plot ({-\y*\y*\y*\y},{3*\y*\y*\y});
\node[anchor=west] at (0,0) {$p$};
\draw[fill=black] (0,0) circle (.045);
\node[anchor=north] at (0,-1.6) {\small(c)};
\end{scope}
\end{tikzpicture}
\caption{
The curve in (a) is $C^{\infty}$. The curves in (b) and (c) are $C^1$ and piece-wise $C^2$ with a unique discontinuity at $p$. The upper and lower curvatures at $p$ for the curve in (b) are finite. The curve in (c) is given by the spherical projection of the plane curve $t\mapsto(-t^\frac{4}{3},t)$. Note that there does not exist a circle tangent from left to this curve at $p=(0,0)$. The upper and lower curvatures at $(0,0)$ are both $+\infty$.}
\label{fig:broader}
\end{figure}

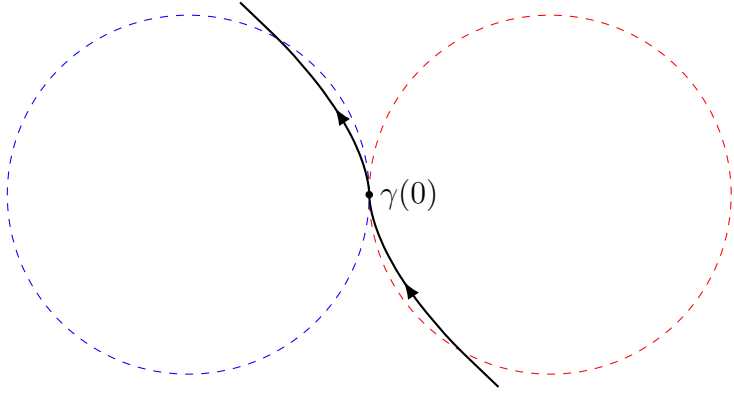
\begin{figure}[t]
\begin{tikzpicture}
\begin{scope}[scale=1.7, xshift = 5cm, decoration={markings, mark=at position 0.3 with {\arrow {Latex}}, mark=at position 0.7 with {\arrow {Latex}}}]
\draw[dashed,red] (1.4,0) circle (1.4);
\draw[dashed,blue] (-1.4,0) circle (1.4);
\draw[postaction={decorate},scale=0.5,domain=-1:1,smooth,variable=\y,thick]  plot ({-2*\y*\y*\y*\y*\y},{3*\y*\y*\y});
\node[anchor=west] at (0,0) {$\gamma(0)$};
\draw[fill=black] (0,0) circle (.025);
\end{scope}
\end{tikzpicture}
\caption{
The curve is the image of the spherical projection of the plane curve $\gamma(t)$ given by $t\mapsto (-t^\frac{5}{3},t)$. Note that there is neither a circle tangent from left nor from right to the curve at $(0,0)$. So, by definition, the upper and the lower curvatures in the point of inflection $\gamma(0)$ are $+\infty$ and $-\infty$, respectively.}
\label{fig:ctex}
\end{figure}

Following the same idea, one may define $C^2$ regular curves tangent from left and from right, and 
give alternative  definition of upper and lower curvatures for $C^1$-immersed curves. It is straightforward that the above definition of the curvatures are equivalent to:
\begin{align*}
\kappa_\gamma^+(t) &= \inf\left\{\kappa_\alpha(s);\text{where $\alpha$ is a $C^2$ regular curve tangent from left to $\gamma$ at $\alpha(s)=\gamma(t)$} \right\}.\\
\kappa_\gamma^-(t) &= \sup\left\{\kappa_\alpha(s);\text{where $\alpha$ is a $C^2$ regular curve tangent from right to $\gamma$ at $\alpha(s)=\gamma(t)$} \right\}.
\end{align*}
Definition \ref{definition4} is motivated by the concepts of upper and lower derivatives of continuous functions in Calculus. 
 For $C^2$ regular curves, upper and lower curvatures are equal to the \emph{geodesic curvature} which is defined by Equation \eqref{curv}. 

Now we take our attention on the curves in $\mathcal{I}$ which start and end at given points with given directions.   The \emph{Frenet frame} of $\gamma$ is defined by:
\begin{equation}\label{frenet}\mathfrak{F}_{\gamma}(t)=\left( \begin{array}{ccc} | & | & | \\
\gamma(t) & \vt_\gamma (t) & \bsm{n}_\gamma(t) \\
| & | & | \end{array} \right) \in \textrm{SO}_3(\mathbb{R}) .
\end{equation}

The space $\SO$ is homeomorphic to the unit tangent bundle of sphere $\textup{UT}\mathbb{S}^2$ by mapping the matrix $\bsm{M}\in\SO$ to the vector $\big(\mat{M}(1,0,0),\bsm{M}(0,1,0)\big) \in \textup{T}_{\bsm{M}(1,0,0)}\mathbb{S}^2$. Let $\mat{I}$ be the identity matrix in $\SO$. We define the following spaces of curves:

\begin{definition}\label{deflspace}
Given $\mP,\mQ\in \textrm{SO}_3(\mathbb{R})$, $\kappa_1,\kappa_2 \in [-\infty,+\infty]$, with $\kappa_1\leq\kappa_2$, we give the following notations of the sets:
\begin{itemize}
\item $\mathcal{I}(\mP,\mQ)$ denotes  the set of all $C^1$ regular curves in $\mathbb{S}^2$ with Frenet frames $\mathfrak{F}_\gamma(0)=\mP$ and $\mathfrak{F}_\gamma(1)=\mQ$. 
\item  $\mathcal{L}^{\kappa_2}_{\kappa_1}(\mP,\mQ)\subset\mathcal{I}(\mP,\mQ)$ denotes  subset of curves that satisfies $\kappa_1<\kappa_\gamma^-(t)\leq\kappa_\gamma^+(t)<\kappa_2$ for all $t \in [0,1]$. 
\item  $\mathcal{S}_{\kappa_1}^{\kappa_2}(\mat{P},\mat{Q})$ denotes the subset  of $\mathcal{L}_{\kappa_1}^{\kappa_2}(\mat{P},\mat{Q})$, in which every curve satisfies  $$\displaystyle\kappa_1<\inf_{t\in [0,1]}\kappa_\gamma^-(t)\leq \sup_{t\in [0,1]}\kappa_\gamma^+(t) < \kappa_2.$$
\end{itemize} 
\end{definition}

 As  stated in Introduction, let  $\gamma(t)$, $t\in [0,1]$, be a $C^1$ regular curve.  Suppose that $\dot{\gamma}(t)$ is Lipschitz continuous. Then it is known that $\ddot{\gamma}(t)$ exists for a.e. $t$.
With reparameterization with arc-length $s$, we have $\gamma'(s)$ is Lipschitz continuous for $s$. This implies that  $\gamma''(s)={\vet{t}}'(s)$ exists for a.e. $s$ and
$${\vet{t}}'(s)=-\gamma(s)+\kappa(s){\vet{n}}(s), \quad a.e. \quad s.$$
In this article, we study the following subset of $\mathcal{I}(\mP,\mQ)$:
\begin{definition}Let $\mathcal{P}_{\kappa_1}^{\kappa_2}(\mP, \mQ)$ be the subset of all curves 
$\gamma$ in $\mathcal{I}(\mP,\mQ)$ whose $\dot{\gamma}$ is Lipschitz continuous and whose geodesic curvature $\kappa(t)$ at a.e. $t$ satisfies 
$$\kappa_1<\essinf_{t\in[0,1]} \kappa(t)\leq \esssup_{t\in[0,1]} \kappa(t)<\kappa_2.$$
\end{definition}

\section{Local behavior of the curves in $\mathcal{L}_{\kappa_1}^{\kappa_2}(\mat{P},\mat{Q})$}\label{local}
The constraints to the lower and upper curvatures of a curve in $\mathcal{L}_{\kappa_1}^{\kappa_2}(\mat{P},\mat{Q})$  influences the local behavior of the curve which is stated in Lemma \ref{sandwich}. The geometrical intuition of this lemma is shown in Figure \ref{fig:lemma}. Roughly speaking,  the curve $\gamma(s)$ doesn't contact the circle tangent to $\gamma(s_0)$ from left with the radius $\rho_2=\arccot(\kappa_2)$ again for $s\in(s_0-\delta,s_0+\delta)$ for some $\delta>0$. The analogous property happens on the circle tangent to $\gamma(s_0)$ from right with the radius $\rho_1=\arccot(\kappa_1)$. More precisely, we prove that

\begin{figure}[b]
\begin{tikzpicture}
\def\r{1}
\def\s{1.4}
\def\a{114.591+90}
\def\b{114.591+90}
\def\c{-81.851-90}
\def\d{81.851-90}
\begin{scope}[scale=1.7, xshift = 5cm, decoration={markings, mark=at position 0.3 with {\arrow {Latex}}, mark=at position 0.7 with {\arrow {Latex}}}]
\draw[green!30,fill=green!30] 
({\r*(cos(deg(2)-\a)+2*sin(deg(2)-\a))}, {-1+\r*(sin(deg(2)-\a)-2*cos(deg(2)-\a))}) --
({\r*(cos(deg(0)-\a)+0*sin(deg(0)-\a))}, {-1+\r*(sin(deg(0)-\a)-0*cos(deg(0)-\a))}) --
({\r*(cos(deg(0)+\b)+0*sin(deg(0)+\b))}, {-1+\r*(sin(deg(0)+\b)-0*cos(deg(0)+\b))}) --
({\r*(cos(deg(-2)+\b)-2*sin(deg(-2)+\b))}, {-1+\r*(sin(deg(-2)+\b)+2*cos(deg(-2)+\b))}) --
({\s*(cos(deg(0)+\c)+0*sin(deg(0)+\c))}, {1.4+\s*(sin(deg(0)+\c)-0*cos(deg(0)+\c))}) --
({\s*(cos(deg(0)+\d)+0*sin(deg(0)+\d))}, {1.4+\s*(sin(deg(0)+\d)-0*cos(deg(0)+\d))}) -- cycle;
\draw[dashed,blue, domain=0:2, variable=\t, samples=100,fill=green!30] 
        plot[fixed point arithmetic] ({\r*(cos(deg(\t)-\a)+\t*sin(deg(\t)-\a))}, {-1+\r*(sin(deg(\t)-\a)-\t*cos(deg(\t)-\a))});
\draw[dashed,blue, domain=0:-2, variable=\t, samples=100,fill=green!30] 
        plot[fixed point arithmetic] ({\r*(cos(deg(\t)+\b)+\t*sin(deg(\t)+\b))}, {-1+\r*(sin(deg(\t)+\b)-\t*cos(deg(\t)+\b))});
\draw[dashed,red, domain=0:2*0.715, variable=\t, samples=100,fill=green!30] 
        plot[fixed point arithmetic] ({\s*(cos(deg(\t)+\c)+\t*sin(deg(\t)+\c))}, {1.4+\s*(sin(deg(\t)+\c)-\t*cos(deg(\t)+\c))});
\draw[dashed,red, domain=0:-2*0.715, variable=\t, samples=100,fill=green!30] 
        plot[fixed point arithmetic] ({\s*(cos(deg(\t)+\d)+\t*sin(deg(\t)+\d))}, {1.4+\s*(sin(deg(\t)+\d)-\t*cos(deg(\t)+\d))});
\draw[dashed,red,fill=white] (0,1.4) circle (1.4);
\draw[dashed,blue,fill=white] (0,-1) circle (1);
\draw[postaction={decorate},scale=1.5,domain=-1:1,smooth,variable=\y,thick] (0,0) .. controls (1,0) and (1.2,-0.1) .. (1.95,0.5);
\draw[postaction={decorate},scale=1.5,domain=-1:1,smooth,variable=\y,thick] (-2.0,-0.2) .. controls (-1.2,0.5) and (-0.76,0) .. (0,0);
\node[anchor=south] at (0,0) {$\gamma(s_0)$};
\draw[fill=black] (0,0) circle (.025);
\node[anchor=south] at (0,1.4) {$\vet{v}_2$};
\draw[fill=black] (0,1.4) circle (.025);
\node[anchor=south] at (0,-1) {$-\vet{v}_1$};
\draw[fill=black] (0,-1) circle (.025);
\end{scope}
\end{tikzpicture}
\caption{Intuitively, the statement \eqref{circineq1} implies that for $s$ in the suitable interval (for instance, given by Lemma \ref{sandwich}), the curve $\gamma$ is sandwiched into the green region delimited by the involute of the tangent circles that have radius $\rho_2$ and $\rho_1$ centered at $\vet{v}_2$ and $\vet{v}_1$ respectively. The second inequality of the statements \eqref{circineq2} and \eqref{circineq3} mean that the distance from $\vet{v}_2$ to $\gamma(s)$ is non-decreasing as $s$ grows  in the respective interval. In the other words, the tangent vector $\vet{t}_\gamma(s)$ lies in the South hemisphere of the sphere with the North direction pointed to $\vet{v}_2$.}
\label{fig:lemma}
\end{figure}
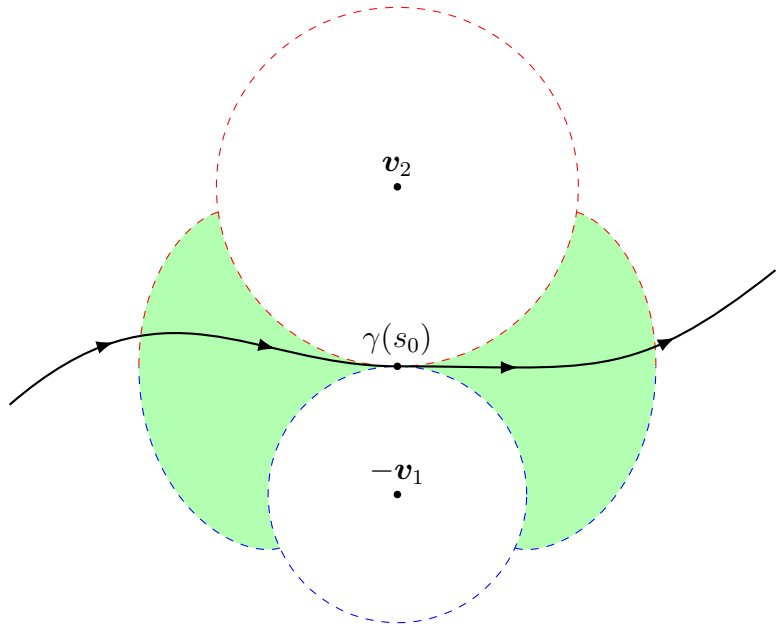

\begin{lemma}\label{sandwich} Consider $-\infty < \kappa_1 < \kappa_2 < +\infty $. Let $\gamma\in\mathcal{L}_{\kappa_1}^{\kappa_2}(\mP,\mQ)$, $\gamma: [0,L_\gamma] \to \mathbb{S}^2$ is parameterized by its arc-length. Let $s_0\in [0,L_\gamma]$, $\rho_1\coloneqq\arccot(\kappa_1)$, $\rho_2\coloneqq\arccot(\kappa_2)$, $\vet{v}_1\coloneqq(\cos\rho_1)\gamma(s_0)+(\sin\rho_1)\vet{n}_\gamma(s_0)$ and $\vet{v}_2\coloneqq(\cos\rho_2)\gamma(s_0)+(\sin\rho_2)\vet{n}_\gamma(s_0)$. Moreover, let $\delta=\min\left\{2\sin\rho_1,2\sin\rho_2\right\}$ and $\bar\delta=\min\left\{\frac{\pi}{2}\sin\rho_1,\frac{\pi}{2}\sin\rho_2\right\}$. Then
\begin{equation}\label{circineq1}
d\left(\vet{v}_1,\gamma(s)\right)\leq\rho_1 \quad\text{and}\quad d\left(\vet{v}_2,\gamma(s)\right)\geq\rho_2,\quad\text{for all }s\in\left[s_0-\delta,s_0+\delta\right]\cap [0,L_\gamma].
\end{equation}
Furthermore,
\begin{equation}\label{circineq2}
\left\langle \vet{t}_\gamma(s),\vet{v}_1 \right\rangle \geq 0\quad\text{and}\quad\left\langle \vet{t}_\gamma(s),\vet{v}_2 \right\rangle \leq 0,\quad\text{for all } s\in\left[s_0,s_0+\bar{\delta}\right]\cap [0,L_\gamma],
\end{equation}
and
\begin{equation}\label{circineq3}
\left\langle \vet{t}_\gamma(s),\vet{v}_1 \right\rangle \leq 0\quad\text{and}\quad\left\langle \vet{t}_\gamma(s),\vet{v}_2 \right\rangle \geq 0,\quad\text{for all } s\in\left[s_0-\bar{\delta},s_0\right]\cap [0,L_\gamma]. 
\end{equation}
\end{lemma}

\begin{proof} We denote $\vet{v}=\vet{v}_2$. Up to a rotation and a change of parameterization, we assume without loss of generality that 
\begin{equation*}
s_0=0, \quad \gamma(0)=(1,0,0) \quad\text{and}\quad \vet{t}_\gamma(0)=(0,1,0).
\end{equation*}
We present the demonstration only for the inequalities for $\vet{v}_2$ on each case. The other inequalities for $\vet{v}_1$ hold by analogous demonstration. 



Moreover, we restrict the proof of the conclusion \eqref{circineq1} for positive $s$ values ($s\in [0,\delta]\cap I$). For $s$ negative ($s\in [-\delta,0]\cap I$) the demonstration is analogous. 

Suppose, by contradiction, that there exists some $\bar{s}\in [0,\delta]\cap I$ such that $d(\vet{v},\gamma(\bar{s}))<\rho_2$.  Then take
\begin{equation*}
\bar{s}_0\coloneqq\inf\left\{s\in [0,\delta]\cap I ; d(\vet{v},\gamma(s))<\rho_2\right\}.
\end{equation*}
{\bf Claim 1.}  There exists a $\sigma>0$ very small such that 
\begin{equation*}
d(\vet{v},\gamma(s))\geq\rho_2,\quad\text{for all $s\in (-\sigma,\sigma)\cap I$.}
\end{equation*}
Since $\kappa^{+}_{\gamma}(0)<\kappa_2=\cot\rho_2$, there exists a $r$ such that $\rho_2<r<\arccot \kappa^{+}_{\gamma}(0) $. Then there is a circle in $\mathbb{S}^2$ with the center $a$ and the radius $r$ which is tangent to $\gamma$ at $\gamma(0)$ from the left. So $$d(a, \gamma(s)\geq r>\rho_2, \quad s\in (-\sigma,\sigma)\cap I.$$
Note that $a$, $\vet{v}$ and $\gamma(0)$ are in the same great circle. So $d(a,\vet{v})=d(a,\gamma(0))-d(\vet{v},\gamma(0))=r-\rho_2$. This implies that
$$d(\vet{v}, \gamma(s))\geq  d(a, \gamma(s))-d(a,\vet{v})\geq r-(r-\rho_2)=\rho_2.$$
Thus we have proved Claim 1. 

By Claim 1, $\bar{s}_0>0$. Moreover
\begin{equation*}
d(\vet{v},\gamma(\bar{s}_0))=\lim_{s\to\bar{s}_0}d(\vet{v},\gamma(s)) = \rho_2.
\end{equation*}
For each $\tau\in\mathbb{R}$, we consider the circle $\zeta_\tau$ with radius $\rho_2$ centered at 
\begin{equation*}
\vet{v}_\tau\coloneqq\left(\cos\rho_2,(\sin\rho_2)(\sin\tau),(\sin\rho_2)(\cos\tau)\right).
\end{equation*}
 Note that, in particular, $\vet{v}_0 = \vet{v}$. By continuity, the intersection of the curve $\gamma$ with each $\zeta_\tau$ consists in at least two points for $\tau$ sufficiently small (see Figure \ref{fig:lemma2}). It is straightforward that if $\langle(0,1,0),\gamma(\bar{s}_0)\rangle\leq 0$ then $s_0 > \min\{2\sin\rho_1,2\sin\rho_2\}$, so Inequality \eqref{circineq1} holds. So we assume $\langle(0,1,0),\gamma(\bar{s}_0)\rangle > 0$.

\begin{figure}[H]
\begin{tikzpicture}
\def\r{1}
\def\s{1.4}
\def\u{20.3}
\begin{scope}[scale=1.7, xshift = -2.1cm, decoration={markings, mark=at position 0.3 with {\arrow {Latex}}, mark=at position 0.7 with {\arrow {Latex}}}]
\draw[dashed,red] (0,1.4) circle (1.4);
\draw[postaction={decorate},scale=1.5,domain=-1:1,smooth,variable=\y,thick] (0,0) .. controls (1.5,0) and (1.6,1.3) .. (0.55,1.5);
\node[anchor=south] at (0,0) {$\gamma(0)$};
\draw[fill=black] (0,0) circle (.025);
\node[anchor=south] at (0,1.4) {$\vet{v}$};
\draw[fill=black] (0,1.4) circle (.025);
\node[anchor=south west] at (1.19,2.14) {$\gamma(\bar{s}_0)$};
\draw[fill=black] (1.19,2.14) circle (.025);
\end{scope}
\begin{scope}[scale=1.7, xshift = 2.1cm, decoration={markings, mark=at position 0.3 with {\arrow {Latex}}, mark=at position 0.7 with {\arrow {Latex}}}]
\draw[dashed,red] ({1.4*sin(\u)},{1.4*cos(\u)}) circle (1.4);
\draw[postaction={decorate},scale=1.5,domain=-1:1,smooth,variable=\y,thick] (0,0) .. controls (1.5,0) and (1.6,1.3) .. (0.55,1.5);
\node[anchor=south] at (0,0) {$\gamma(0)$};
\draw[fill=black] (0,0) circle (.025);
\node[anchor=south] at ({1.4*sin(\u)},{1.4*cos(\u)}) {$\vet{v}_{\tau_0}$};
\draw[fill=black] ({1.4*sin(\u)},{1.4*cos(\u)}) circle (.025);
\node[anchor=west] at (1.85,1.01) {$\gamma(s_{\tau_0})$};
\draw[fill=black] (1.85,1.01) circle (.025);
\end{scope}
\end{tikzpicture}
\caption{
The intersections between $\zeta_\tau$ (dashed red circles) and $\gamma$ are shown in the images above.}
\label{fig:lemma2}
\end{figure}
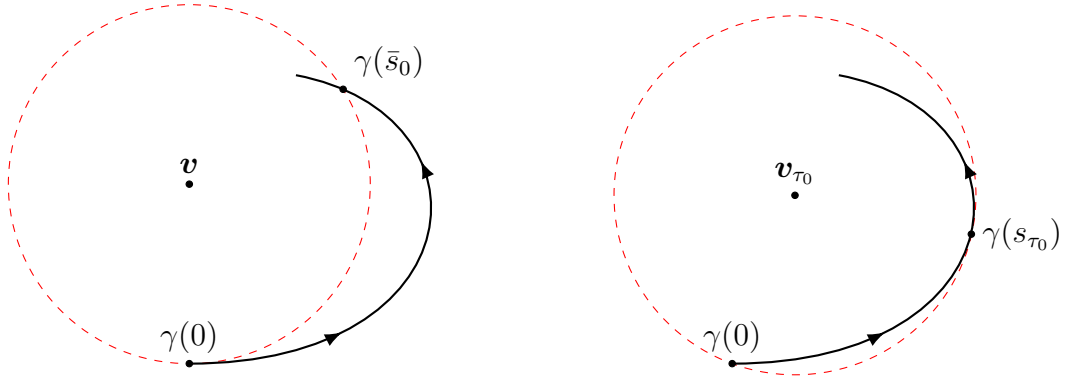

Moreover, there exists a $\tau_0>0$ such that for all $\tau\in[0,\tau_0]$ there exists a $s_\tau$ satisfying $d(\vet{v}_\tau,\gamma(s_\tau))=\rho_2$. Furthermore, $s_\tau=s(\tau)$ may be chosen as a strictly decreasing function and $\tau_0$ may be chosen as the smallest number satisfying such property. 

Since $s>\min\{2\sin(\rho_1),2\sin(\rho_2)\}$, $\gamma(s_{\tau_0})$ and $\gamma(0)$ are not two opposite points of the circle $\zeta_{\tau_0}$ nor $\gamma(0)=(1,0,0)=\gamma(s_{\tau_0})$. So evidently the tangent vector of the curve $\vet{t}_\gamma(s_{\tau_0})$ is also tangent to the circle $\zeta_{\tau_0}$ at $\gamma(s_{\tau_0})$ and $d(\vet{v}_{\tau_0},\gamma(s))\leq\rho_2$ for all $s\in(s_{\tau_0}-\bar{\epsilon}, s_{\tau_0}+\bar{\epsilon})$ for an $\bar{\epsilon}>0$ sufficiently small. Thus $\kappa_\gamma^+(s_{\tau_0})\geq\cot \rho_2 = \kappa_2 $ contradicting $\gamma\in\mathcal{L}_{\kappa_1}^{\kappa_2}(P,Q)$. So Inequality \eqref{circineq1} holds.

\vspace{.5em} 

Now we  prove \eqref{circineq2} using \eqref{circineq1}.   Suppose, by contradiction, that there exists some $\bar{s}\in\left[0,\bar{\delta}\right]$ such that $\left\langle \vet{t}_\gamma(\bar{s})-\vet{t}_\gamma(0),\vet{v} \right\rangle > 0$. By the continuity of $\vet{t}_\gamma$ there exists a $\bar{s}_0\in [0,\bar{s})$ such that:
\begin{equation*}
\left\langle \vet{t}_\gamma(s)-\vet{t}_\gamma(0),\vet{v} \right\rangle > 0 \quad \text{for all $s\in(\bar{s}_0,\bar{s}]$.}
\end{equation*}
Moreover $\bar{s}_0$ may be taken to be the smallest number satisfying such property.  By continuity:
\begin{equation*}
\left\langle \vet{t}_\gamma(\bar{s}_0)-\vet{t}_\gamma(0),\vet{v} \right\rangle = 0.
\end{equation*}
We shall deduce that $\bar{s}_0\geq\min\left\{\frac{\pi}{2}\sin\rho_2,2\sin\rho_1\right\}$ which leads to a contradiction. 

Firstly, we assert that $\bar{s}_0>0$. Otherwise if $\bar{s}_0=0$ then, since $\left\langle \vet{t}_\gamma(s),\vet{v} \right\rangle > 0$ for $s\in [0,\bar{s}]$,
\begin{equation*}
\frac{d}{ds}\big(d(\vet{v},\gamma(s))\big)=\frac{d}{ds}\arccos\big(\left\langle \vet{v},\gamma(s) \right\rangle\big) = -\frac{1}{\big(1-\langle\vet{v},\vet{t}_\gamma(s)\rangle^2\big)^{\frac{1}{2}}}\langle\vet{v},\vet{t}_\gamma(s)\rangle < 0.
\end{equation*}
This implies $d(\vet{v},\gamma(s))<d(\vet{v},\gamma(0))=\rho_2=\arccot\kappa_2$ for $s>0$ sufficiently small. Hence $\kappa^+_\gamma(0)\geq\kappa_2$, contradicting $\gamma\in\mathcal{L}_{\kappa_1}^{\kappa_2}(P,Q)$. Therefore $\bar{s}_0>0$.

Note that the set $\{\gamma(\bar{s}_0),\vet{t}_\gamma(\bar{s}_0),\vet{v}\}$ is linearly independent. Otherwise $\vet{v}=-\gamma(\bar{s}_0)$ (note that $\vet{v}\neq\gamma(\bar{s}_0)$ from the relation \eqref{circineq1}).  This implies that $\bar{s}_0\geq \min\{2\sin\rho_1,2\sin\rho_2\}$, which  does not happen by the definition of $\bar{s}_0$. So the basis $B\coloneqq\{\gamma(s),\vet{t}_\gamma(s),\vet{v}\}$ is either positive or negative. If $B$ is negative, then again $\bar{s}_0\geq \min\{2\sin\rho_1,2\sin\rho_2\}$, contradicting the definition of $\bar{s}_0$.

Next, we consider the great circle connecting $\vet{v}$ and $\gamma(\bar{s}_0)$. Let $\tilde{\vet{v}}$ be the center of the tangent circle of radius $\rho_2$ at $\gamma(\bar{s}_0)$. Then $\vet{v}$, $\tilde{\vet{v}}$ and $\gamma(\bar{s}_0)$ are in the same great circle. By Inequality \eqref{circineq1}, $d(\gamma(0),\tilde{\vet{v}})>\rho_2$. If the angle $\angle\gamma(0)\vet{v}\gamma(\bar{s}_0)$ is greater than $\frac{\pi}{2}$, then the arc from $\gamma(0)$ to $\gamma(\bar{s}_0)$ is greater than $\frac{\pi}{2}\sin(\rho_2)$. That is $\bar{s}_0>\frac{\pi}{2}\sin(\rho_2)$ contradicting the hypothesis  $\bar{s}_0\in \left[0,\frac{\pi}{2}\sin\rho_2\right]$. So $\angle\gamma(0)\vet{v}\gamma(\bar{s}_0)\leq\frac{\pi}{2}$.  Since $\tilde{\vet{v}}$ is on the great circle connecting $\vet{v}$ and $\gamma(\bar{s}_0)$. Consider a point $\vet{u}$ in this great circle such that the triangle $\triangle \gamma(0)\vet{v}\vet{u}$ is an isosceles triangle with $\angle\vet{v} = \angle\vet{u}$. Since $d(\tilde{\vet{v}},\gamma(0))\geq\rho_2$ (by \eqref{circineq1}), $\tilde{\vet{v}}$ lies outside of the segment $\vet{v}\vet{u}$ (in the triangle $\triangle \gamma(0)\vet{v}\vet{u}$). The previous assertion and the positivity of the basis $B\coloneqq\{\gamma(s),\vet{t}_\gamma(s),\vet{v}\}$ imply that 
\begin{equation*}
\angle\gamma(0)\tilde{\vet{v}}\gamma(\bar{s}_0)\geq \angle\gamma(0)\vet{u}\gamma(\bar{s}_0) = \pi - \angle\gamma(0)\vet{v}\gamma(\bar{s}_0) \geq \frac{\pi}{2}.
\end{equation*}
So again the arc from $\gamma(0)$ to $\gamma(\bar{s}_0)$ is greater than $\frac{\pi}{2}\sin(\rho_2)$. That is $\bar{s}_0>\frac{\pi}{2}\sin(\rho_2)$ contradicting the hypothesis  $\bar{s}_0\in \left[0,\frac{\pi}{2}\sin\rho_2\right]$. The proof of \eqref{circineq2} is complete.

The proof of the statement \eqref{circineq3} is analogous to \eqref{circineq2}. 
\end{proof}

\begin{remark}Observe  that the intervals of Lemma \ref{sandwich} are not optimal, a sharper result can be obtained with a more careful demonstration, but the current one is enough for the applications in this article.
\end{remark}

\section{Geometry of the curves in  $\mathcal{P}_{\kappa_1}^{\kappa_2}(\mat{P},\mat{Q})$}\label{geometry}

In this section we show that  the sets  $\mathcal{P}_{\kappa_1}^{\kappa_2}(\mat{P},\mat{Q})$ and  $\mathcal{S}_{\kappa_1}^{\kappa_2}(\mat{P},\mat{Q})$  are the same.

 \begin{theorem}\label{thm1-1}$\mathcal{P}_{\kappa_1}^{\kappa_2}(\mat{P},\mat{Q})=\mathcal{S}_{\kappa_1}^{\kappa_2}(\mat{P},\mat{Q})$
 \end{theorem}

\begin{proof} {\bf Part 1}.  We show that $\mathcal{P}_{\kappa_1}^{\kappa_2}(\mat{P},\mat{Q})\subseteq\mathcal{S}_{\kappa_1}^{\kappa_2}(\mat{P},\mat{Q})$. 

Let  $t_0\in [0,1]$. 
We let $\gamma$ be parameterized by arc-length $s$. So $s_0=s(t_0)$. We write, for simplicity, that $s_0=0$.
For  a $\rho\in (0,\pi)$, we consider the circle $\zeta$ in $\mathbb{S}^2$   of radius $\rho$ with  the center  given by 
$$\vet{v}_1 = \gamma(0)\cos\rho + \vet{n}_\gamma(0)\sin\rho .$$
The circle $\zeta$ satisfies that  $\vet{t}_\gamma'=\vet{t}_\zeta'$ at $\gamma(0)$.
 
We will determine  a value $\rho<\kappa^+(t_0)$ such that $\zeta$ is a left tangent circle of $\gamma$ at $\gamma(0)$. To do it,  we define  the function
$$ g_1(s)=\langle\gamma(s),\vet{v}_1\rangle.$$
{\bf Claim}.  $g_1$ has a local maximum at  $s=0$ for some $\rho\in (0,\pi)$. 

Firstly, from the definition $g_1'(0)=\langle\vet{t}_\gamma(0),\vet{v}_1\rangle =0$.  Now we give estimation on the second order variation at $s=0$:
$$\frac{g'_1(s)-g'_1(0)}{s} = \frac{1}{s}\left\langle\vet{t}_\gamma(s)-\vet{t}_\gamma(0),\vet{v}_1 \right\rangle .$$

 Since $\gamma'(s)={\vet{t}}_{\gamma}(s)$ is in $W^{1,\infty}[0, L_{\gamma}]$ and hence absolutely continuous, we have that
 $$ \vet{t}_\gamma(s) - \vet{t}_\gamma(0) = \int_0^s\vet{t}'_\gamma(u)\, du. $$
 In the above, $\vet{t}'_\gamma$ is defined for a.e. $s\in [0,L_\gamma]$.
 
By $\vet{t}'_\gamma(s)=-\gamma(s)+\kappa(s)\vet{n}_\gamma(s)$ for a.e. $s\in [0,L_\gamma]$,

\begin{equation}\label{eqt}
 \vet{t}_\gamma(s) - \vet{t}_\gamma(0)  = \int_0^s\left(-\gamma(u)+\kappa(s)\vet{n}_\gamma(u)\right)\, du .
\end{equation}

Recall the Taylor's formula, 
$$\gamma(s)=\gamma(0)+O(s).$$
$$\vet{t}_\gamma(s)=\vet{t}_\gamma(0)+O(s).$$
Here the notation of the big $O$  is  used to denote a map such that
$$ \limsup_{s\to 0}\frac{\|O(s)\|}{s} < \infty .$$
Then
\begin{align*}
\vet{n}_\gamma(s) & = \gamma(s)\times\vet{t}_\gamma(s) \\
& = \big(\gamma(0)+O(s)\big)\times\big(\vet{t}_\gamma(0)+O(s)\big)\\
& = \vet{n}_\gamma(0) + O(s).
\end{align*}
Substituting the expressions into Equation \eqref{eqt} we obtain:
\begin{align*}
\vet{t}_\gamma(s)-\vet{t}_\gamma(0) & = \int_0^s \left(-\gamma(0)+O(u)+\kappa(u)\vet{n}_\gamma(0)+\kappa(u)O(u)\right)\, du \\
& = \int_0^s \left(-\gamma(0)+\kappa(u)\vet{n}_\gamma(0)+O(u)\right)\,du
\end{align*}
Here we use the boundedness of $\displaystyle \kappa(s)$ for a.e. $s$.
Thus
\begin{align}\nonumber
\frac{g_1'(s)-g_1'(0)}{s} & = \frac{1}{s}\left\langle\vet{t}_\gamma(s)-\vet{t}_\gamma(0),\vet{v}_1\right\rangle \\ \nonumber
& = \frac{1}{s}\left\langle \int_0^s \left(-\gamma(0)+\kappa(u)\vet{n}_\gamma(0)+O(u)\right)du , \gamma(0)\cos\rho+\vet{n}_\gamma(0)\sin\rho \right\rangle \\ \label{eqlast}
& = \frac{1}{s}\left(\int_0^s \left(\kappa(u)\sin\rho-\cos\rho\right)\,du + \left\langle \int_0^s O(u)\, du,\vet{v}_1 \right\rangle \right). 
\end{align} 
We have that
$$\lim_{s\to 0}\frac{1}{s}\int_0^s O(u)du = 0.$$
By the definition of $\mathcal{P}_{\kappa_1}^{\kappa_2}(\mP,\mQ)$,   there exists a $\bar{\kappa}_2<\kappa_2$ such that $\kappa(s)<\bar{\kappa}_2$ for a.e. $s$.

Taking $\rho=\arccot\left(\frac{\kappa_2+\bar{\kappa}_2}{2}\right)$ in Equation \eqref{eqlast}, the first integral in the right side of \eqref{eqlast} is  strictly small than a negative number by the following computation: When $s>0$,
\begin{align*}
 \frac{1}{s}\int_0^s \left(\kappa(u)\sin\rho-\cos\rho\right)du&=\frac1s\int_0^s(\sin \rho)\left(\kappa(u)-\cot\rho\right)du\\
 &=\frac1s\int_0^s(\sin \rho)\left(\kappa(u)-\frac{\kappa_2+\bar{\kappa}_2}{2}\right)du\\
&\leq \frac{\sin \rho} s\int_0^s\left(-\frac{\kappa_2-\bar{\kappa}_2}{2}\right)du=-\frac{\kappa_2-\bar{\kappa}_2}{2}(\sin\rho)<0.
\end{align*}
The case $s<0$ is analogous. 

We conclude that:
$$\limsup_{s\to 0}\left(\frac{g_1'(s)-g_1'(0)}{s}\right) < 0  \quad \text{for the $\rho$ chosen above.}$$
Hence  $g_1$ has a local maximum at  $s=0$. Thus we have verified the claim.

 Since $g_1(\gamma(s)=\langle\gamma(s),\vet{v}_1\rangle=\cos\theta$, where $\theta$ is equal to the distance $d(\gamma(s) ,\vet{v}_1) $ between $\vet{v}_1$ and $\gamma(s)$  in $\mathbb{S}^2$. The claim  implies that $d(\gamma(s) ,\vet{v}_1)$  has a local minimum at $s=0$. Because $d(\gamma(0),\vet{v}_1)=\rho$, the circle with the center at $\vet{v}_1$ radius $\rho$ is a left tangent circle of the curve $\gamma$ at $s=0$. Thus $\kappa_\gamma^+(0)\leq \frac{\kappa_2+\bar{\kappa}_2}{2} < \kappa_2 $. Since $\frac{\kappa_2+\bar{\kappa}_2}{2}$ is independent of $s_0$, $\sup_{s}\kappa_\gamma^+(0)<\kappa_2$.
 
 The proof of  $\kappa_1<\inf_{s}\kappa_\gamma^-(0)$ is analogous. This finishes  Part $1$.

{\bf  Part 2}.  We show that  $\mathcal{S}_{\kappa_1}^{\kappa_2}(\mat{P},\mat{Q})\subseteq\mathcal{P}_{\kappa_1}^{\kappa_2}(\mat{P},\mat{Q})$. 

Given a $C^1$ curve $\gamma\in\mathcal{S}_{\kappa_1}^{\kappa_2}(\mat{P},\mat{Q})$, we assume without loss of generality that it is parameterized by arc-length $s$ and we shall prove that $\vet{t}_\gamma$ is a Lipschitz continuous function. 

For $s_0\in[0,L_\gamma]$, let  $\rho = \arccot{\kappa}_2$, $\vet{v}=\gamma(s_0)\cos\rho+\vet{n}_\gamma(s_0)\sin\rho$. Since $\kappa_\gamma^+(s_0)<\kappa_2$,  the circle with center $\vet{v}$ and radius $\rho$ is tangent to $\gamma$ from the left at $s_0$ and  the function $g(s)=\langle\gamma(s),\vet{v}_1\rangle$  has a local maximum at $s_0$.  
 Morover, by Lemma  \ref{sandwich},
\begin{align}\label{derivative} 
&\limsup_{h\to 0}\left\langle \frac{\vet{t}_\gamma(s_0+h)-\vet{t}_\gamma(s_0)}{h},\vet{v}\right\rangle=\limsup_{h\to 0}\left\langle \frac{\vet{t}_\gamma(s_0+h)}{h},\vet{v}\right\rangle
 \leq 0.
\end{align}

$\langle \vet{t}_\gamma(s),\vet{t}_\gamma(s) \rangle = \|\gamma'(s)\|^2 = 1$ implies that
\begin{equation}\label{eqlip3}
\lim_{h\to 0}\left\langle \frac{\vet{t}_\gamma(s_0+h)-\vet{t}_\gamma(s_0)}{h},\vet{t}_\gamma(s_0) \right\rangle = 0. 
\end{equation}
By $ \langle \vet{t}_\gamma(s),\gamma(s) \rangle =0$,
\begin{align*}
0&=\lim_{h\to 0}\frac{\left\langle\vet{t}_\gamma(s_0+h),\gamma(s_0+h)\rangle-\langle\vet{t}_\gamma(s_0),\gamma(s_0)\right\rangle}{h}\\
 &=\lim_{h\to 0}\frac{\langle\vet{t}_\gamma(s_0+h)-\vet{t}_\gamma(s_0),\gamma(s_0)\rangle}{h}+\lim_{h\to 0}\frac{\langle\vet{t}_\gamma(s_0+h),\gamma(s_0+h)-\gamma(s_0)\rangle}{h}\\
 &=\lim_{h\to 0}\left\langle\frac{\vet{t}_\gamma(s_0+h)-\vet{t}_\gamma(s_0)}{h},\gamma(s_0)\right\rangle +\langle\vet{t}_\gamma(s_0),\vet{t}_\gamma(s_0)\rangle.
\end{align*}
So we obtain
\begin{equation}\label{eqlip4}
\lim_{h\to 0}\left\langle\frac{\vet{t}_\gamma(s_0+h)-\vet{t}_\gamma(s_0)}{h},\gamma(s_0)\right\rangle = -1.
\end{equation}
By  (\ref{derivative}), we have that
\begin{align*}
0\geq &\limsup_{h\to 0}\left\langle \frac{\vet{t}_\gamma(s_0+h)-\vet{t}_\gamma(s_0)}{h},\vet{v}_1\right\rangle\\
& = \limsup_{h\to 0}\left\langle \frac{\vet{t}_\gamma(s_0+h)-\vet{t}_\gamma(s_0)}{h},\gamma(s_0)\cos\rho+\vet{n}_\gamma(s_0)\sin\rho\right\rangle\\
&= -\cos\rho+(\sin\rho)\limsup_{h\to 0}\left\langle \frac{\vet{t}_\gamma(s_0+h)-\vet{t}_\gamma(s_0)}{h},\vet{n}_\gamma(s_0)\right\rangle. 
\end{align*}
In the last equality above, we used (\ref{eqlip4}). Then
\begin{equation}\label{eqlip1}
\limsup_{h\to 0}\left\langle \frac{\vet{t}_\gamma(s_0+h)-\vet{t}_\gamma(s_0)}{h},\vet{n}_\gamma(s_0) \right\rangle \leq \coth\rho= \kappa_2. 
\end{equation}
Analogously, from $\kappa_1<\kappa_\gamma^-(s)$, we deduce:
\begin{equation}\label{eqlip2}
\kappa_1  \leq \liminf_{h\to 0}\left\langle \frac{\vet{t}_\gamma(s_0+h)-\vet{t}_\gamma(s_0)}{h},\vet{n}_\gamma(s_0) \right\rangle . 
\end{equation}

Since $s_0\in[0,L_\gamma]$ is arbitrary in Equations \eqref{eqlip3}, \eqref{eqlip4} and Inequalities \eqref{eqlip1}, \eqref{eqlip2}. We deduce that for each $s\in [0, L_{\gamma}]$, there is a $\delta_s>0$ such that for $h\in (-\delta_s, \delta_s)$, 
\begin{equation}\label{eqlip}\left|\vet{t}_\gamma(s+h)-\vet{t}_\gamma(s)\right|\leq M|h|,
\end{equation}
where $M$ is the uniform constant for all $s$.

Given any $0\leq s_1< s_2\leq L_{\gamma}$,   $[s_1, s_2]$ has a finite open covering $\{(\bar{s}_i-\delta_{\bar{s}_i}, \bar{s}_i)+\delta_{\bar{s}_i}\}$, $i=1, \ldots, m,$ satisfying Equation \eqref{eqlip} for all $h\in (-\delta_{\bar{s}_i}, \delta_{\bar{s}_i})$. Without lost of generality, we may suppose that $m\geq 2$ and these open sets are ordered 
such that $\bar{s}_1=s_1$, $\bar{s}_m=s_2$, $\bar{s}_i< \bar{s}_{i+1}$, $(\bar{s}_i-\delta_{\bar{s}_i}, \bar{s}_i+\delta_{\bar{s}_i})\cap  (\bar{s}_{i+1}-\delta_{\bar{s}_{i+1}}, \bar{s}_{i+1}+\delta_{\bar{s}_{i+1}})\neq \emptyset$, and  $(\bar{s}_i-\delta_{\bar{s}_i}, \bar{s}_i+\delta_{\bar{s}_i})\not\subseteq(\bar{s}_{j}-\delta_{\bar{s}_{j}}, \bar{s}_{j}+\delta_{\bar{s}_{j}})$ for all $i\neq j$.

 Choose $w_0<w_1<\ldots<w_m$ as $w_0=s_1, w_i\in [\bar{s}_i,\bar{s}_i+\delta_{\bar{s}_i})\cap  (\bar{s}_{i+1}-\delta_{\bar{s}_{i+1}},\bar{s}_{i+1}]$ for $i=1, \ldots, m-1,$ and $w_m=s_2$. Then
\begin{align*}
|\vet{t}_\gamma(s_2)-&\vet{t}_\gamma(s_1)| \\
&\leq |\vet{t}_\gamma(w_0)-\vet{t}_\gamma(w_1)|+|\vet{t}_\gamma(w_1)-\vet{t}_\gamma(w_2)|+\ldots+|\vet{t}_\gamma(w_{m-1})-\vet{t}_\gamma(w_m)|\\
&\leq M|w_1-w_0|+M|w_2-w_1|+\ldots+M|w_m-w_{m-1}|\\
&=M(w_1-w_0+w_2-w_1+\ldots+w_m-w_{m-1})\\
&=M(w_m-w_0)\\
&=M|s_2-s_1|.
\end{align*}

Hence $\vet{t}_\gamma(s)$ is  Lipschitz continuous in $[0, L_{\gamma}]$. Thus the derivative $\vet{t}'_\gamma(s)$ exists almost everywhere and
\begin{equation}
\vet{t}_\gamma(s+h)=\vet{t}_\gamma(s)+\int_s^{s+h}\vet{t}'_\gamma(\tau)\,d\tau.
\end{equation}
By the basic property of curves,  there exists a function $\kappa(s), s\in [0, L_{\gamma}]$ such that, for a.e. $s$,
$$ \vet{t}'_\gamma(s)= -\gamma(s)+\kappa(s)\vet{n}_\gamma(s),$$
and the point $s$ where $\vet{t}_\gamma'(s)$ exists, $\kappa(s)=\kappa^+(s)=\kappa^-(s)$ is the curvature of $\gamma(s)$.

By $\sup\kappa^+<\kappa_2$, we have that $\esssup \kappa(s)<\kappa_2$. Analogously, $\kappa_1<\esssup\kappa^-(s)$.

This means that  $\gamma\in \mathcal{P}_{\kappa_1}^{\kappa_2}(\mat{P},\mat{Q})$.

\end{proof}

\section{ $C^0$ and $C^1$ topologies of the space $\mathcal{P}_{\kappa_1}^{\kappa_2}(\mat{P},\mat{Q})$}\label{C-topology}

It is known that the topology induced by the metrics $d^0$ and $d^1$ are not the same in the space $\mathcal{I}(\mat{P},\mat{Q})$. However, once we restrict these metrics to $\mathcal{P}_{\kappa_1}^{\kappa_2}(\mat{P},\mat{Q})$, they become the same.
In this section,   we show that the topologies of the space $\mathcal{P}_{\kappa_1}^{\kappa_2}(\mat{P},\mat{Q})$ with $C^0$ and $C^1$ topology are equivalent.
Before doing this, we prove the following result which also is of independent interest.
\begin{proposition}\label{propconvlength} Given $-\infty<\kappa_1<\kappa_2<+\infty$, let  $\{\alpha_k\}_{k\in\mathbb{N}}$  be a sequence of $C^1$ regular curves in $\mathbb{S}^2$ whose upper and lower curvatures satisfy $\kappa_1<{\kappa}_{\alpha_k}^{-}(t)\leq {\kappa}_{\alpha_k}^{+}(t)<\kappa_2$, $t\in[0,1]$. Assume that $\{\alpha_k\}_{k\in\mathbb{N}}$ converges to a $C^1$ regular curve $\alpha$ in $\mathbb{S}^2$  in $d^0$ metric, where $\alpha$ has bounded upper and lower curvatures.   Let $L_{\alpha_k}$ be the length of $\alpha_k$  for each $k\in\mathbb{N}$  and $L_\alpha$  the length of $\alpha$ respectively.  Then $\displaystyle\lim_{k\to\infty}L_{\alpha_k}=L_\alpha$.
\end{proposition}

\begin{proof} By the equivalence of the metrics $d^0$ and $\bar{d}^0$, $\{\alpha_k\}_{k\in\mathbb{N}}$ converges in $\alpha$ in $\bar{d}^0$ metric. By contradiction, suppose that  $\displaystyle\lim_{k\to\infty}L_{\alpha_k}\neq L_\alpha$. By taking a subsequence, one of cases below happens:
\begin{enumerate}
\item $\displaystyle\lim_{k\to\infty}L_{\alpha_k}=A\neq L_\alpha $, where $A$ is a finite number.
\item $\displaystyle\lim_{k\to\infty}L_{\alpha_k}=\infty$.
\end{enumerate}
 By Theorem \ref{thm1-1},  for $\alpha$ and each $\alpha_k$,  its  tangent vector   is Lipschitz continuous on the compact set $[0,1]$ and hence absolutely continuous.  By Taylor formula with integral remainder,
\begin{align*}
\alpha_k(t)-\alpha_k(0) & =\dot{\alpha}_k(0)t+\int_0^t\ddot{\alpha}_k(\tau)(t-\tau)d\tau
\end{align*}
\begin{align*}
\alpha(t)-\alpha(0)  &=\dot{\alpha}(0)t+\int_0^t\ddot{\alpha}(\tau)(t-\tau)d\tau.
\end{align*}
Note that  the parameter $t\in [0,1]$ satisfies  $s=L_{\alpha_k}t$ for each $k$ and $s=L_{\alpha}t$, where $s$ denotes the corresponding arc-length parameter for $\alpha_k$ and $\alpha$ respectively. We have

\begin{align}\label{equation-difference}
\alpha_k(t)-\alpha(t) - & (\alpha_k(0)-\alpha(0)) \nonumber  \\
& =(\dot{\alpha}_k(0)-\dot{\alpha}(0))t+\int_0^t[\ddot{\alpha}_k(\tau)-\ddot{\alpha}(\tau)](t-\tau)d\tau\nonumber\\
 & =[L_{\alpha_k}\vet{t}_{\alpha_k}(0)-L_{\alpha}\vet{t}_{\alpha}(0)]t+ \int_0^t [L_{\alpha_k}^2\vet{t}'_{\alpha_k}(\tau)-L_{\alpha}^2\vet{t}'_{\alpha}(\tau)](t-\tau)d\tau,
\end{align}
where $\vet{t}'_{\alpha_k}=-\alpha_k +\kappa_{\alpha_k}\vet{n}_{\alpha_k}$, $\vet{t}'_{\alpha}=-\alpha +\kappa_{\alpha}\vet{n}_{\alpha}$ for a.e. $s$, 
$\kappa_{\alpha_k}$ is in $(\kappa_1,\kappa_2) $, and $\kappa_{\alpha}$ is bounded.
So $\vet{t}'_{\alpha_k}$ and $\vet{t}'_{\alpha}$ are uniformly bounded for all $k$ and $t$.
 
 \textbf{Case 1}:
Since $\displaystyle\lim_{k\to\infty}L_{\alpha_k}=A<\infty$, there exists a positive constant $c$ such that 
$$|L_{\alpha_k}^2\vet{t}'_{\alpha_k}(\tau)-L_{\alpha}^2\vet{t}'_{\alpha}(\tau)|\leq c.$$ Hence
\begin{align*}
|\alpha_k(t)-\alpha(t)| 
 & \geq |L_{\alpha_k}\vet{t}_{\alpha_k}(0)-L_{\alpha}\vet{t}_{\alpha}(0)|t-c\int_0^t(t-\tau)d\tau-|\alpha_k(0)-\alpha(0)| \\
 &= |L_{\alpha_k}-L_{\alpha}|t-\frac c2t^2-|\alpha_k(0)-\alpha(0)|.
\end{align*}
Since $A\neq L_{\alpha}$, there is a very small $\epsilon>0$ satisfying $\frac{\epsilon}{c}< 1$ and $k_0\geq 0$ such that for  $k\geq k_0$, $|L_k-L_{\alpha}|\geq \epsilon>0$. We may  choose $k_0$ such that for $k\geq k_0$, $|\alpha_k(0)-\alpha(0)|\leq  \frac{\epsilon^2}{4c}$ also holds. Then, for $0\leq t\leq \frac{\epsilon}{c}$,
\begin{align*}
\left|\alpha_k(t)-\alpha(t)\right| 
 &\geq t\left(\epsilon-\frac{c}{2}t\right)- \frac{\epsilon^2}{2c}\geq \frac{\epsilon}{2}t -\frac{\epsilon^2}{4c}.
\end{align*}
Taking $t=\frac{\epsilon}{c}$, we have that 
\begin{align*}\left|\alpha_k\left(\frac{\epsilon}{c}\right)-\alpha\left(\frac{\epsilon}{c}\right)\right| 
 &\geq  \frac{\epsilon^2}{4c}>0.
 \end{align*}
 So
 \begin{align*}
\max_{t\in[0,1]}|\alpha_k(t)-\alpha(t)|  &\geq   \frac{\epsilon^2}{4c}>0.
\end{align*}
This implies that $\displaystyle\lim_{k\rightarrow \infty}\bar{d}^0(\alpha_k,\alpha)=\lim_{k\rightarrow \infty}\max_{t\in[0,1]}|\alpha_k(t)-\alpha(t)|\neq 0 $ which is a contradiction.

\textbf{Case 2}: 
Similar to Case 1,  there exists  positive constants $c_1$ and $c_2$ such that 
\begin{align*}
|\alpha_k(t)-\alpha(t)| 
 & \geq |L_{\alpha_k}-L_{\alpha}|t-c_1L_{\alpha_k}^2\int_0^t(t-\tau)d\tau-c_2L_{\alpha}^2\int_0^t(t-\tau)d\tau-|\alpha_k(0)-\alpha(0)|\\
 &= |L_{\alpha_k}-L_{\alpha}|t-\frac {c_1L_{\alpha_k}^2+c_2L_{\alpha}^2}2t^2-|\alpha_k(0)-\alpha(0)|.
\end{align*}
Since $L_{\alpha_k}\rightarrow\infty$, there exists a number  $k_0>0$ such that $L_{\alpha_k}\geq \max\{1, 2L_{\alpha}\}$ for $k\geq k_0$. So for $k\geq k_0$,
\begin{align*}
|\alpha_k(t)-\alpha(t)|  &\geq  \frac1{2}\left(L_{\alpha_k}t-cL_{\alpha_k}^2t^2\right)-|\alpha_k(0)-\alpha(0)|,
\end{align*}
where $c=c_1+\frac{c_2}{4}.$
Again choose a number $k_1\geq k_0$ such that for $k\geq k_1$,  $2cL_{\alpha_k}\leq 1$ and $|\alpha_k(0)-\alpha(0)|\leq \frac{1}{16c}.$
For each $k\geq k_1$, we take $t_k=\frac{1}{2cL_{\alpha_k}}\in [0,1]$. Then
\begin{align*}
 \left|\alpha_k(t_k)-\alpha(t_k)\right| \geq \frac{1}{2}\left(\frac{1}{2c}-\frac{1}{4c}\right)-\frac{1}{16c}=\frac1{8c}-\frac1{16c}=\frac1{16c}.
\end{align*}
Thus
\begin{align*}
\max_{t\in[0,1]}\left|\alpha_k(t)-\alpha(t)\right| &\geq \frac1{16c}.
\end{align*}

This  induces a contradiction with $\displaystyle\lim_{k\rightarrow \infty}\bar{d}^0(\alpha_k,\alpha)=0$.
\end{proof}

As a corollary, Proposition \ref{propconvlength} implies the following lemma

\begin{lemma}\label{lemconvlength} Let $-\infty<\kappa_1<\kappa_2<+\infty$. Consider a sequence $\{\alpha_k\}_{k\in\mathbb{N}}$ in $\mathcal{L}_{\kappa_1}^{\kappa_2}(\mat{P},\mat{Q})$ converging to $\alpha\in\bar{\mathcal{L}}_{\kappa_1}^{\kappa_2}(\mat{P},\mat{Q})$ in $d^0$ metric. For each $k\in\mathbb{N}$, let $L_{\alpha_k}$ be the length of $\alpha_k$  and $L_\alpha$  the length of $\alpha$ respectively.  Then $\displaystyle\lim_{k\to\infty}L_{\alpha_k}=L_\alpha$.
\end{lemma}

Now we are ready to prove

\begin{theorem} \label{thm-2-2}Let $-\infty <\kappa_1<\kappa_2<+\infty$ and $\mat{P},\mat{Q}\in\SO$, the metric spaces $(\mathcal{P}_{\kappa_1}^{\kappa_2}(\mat{P},\mat{Q}),d^0)$ and $(\mathcal{P}_{\kappa_1}^{\kappa_2}(\mat{P},\mat{Q}),d^1)$ generate the same topology.
\end{theorem}

\begin{proof}  For a curve $\alpha\in\mathcal{P}_{\kappa_1}^{\kappa_2}(\mat{P},\mat{Q})$, we choose here the parameter $t$ so that it is proportional to arc-length and $\alpha:[0,1]\to\mathbb{S}^2$ has constant speed $|\dot{\alpha}|\equiv L_{\alpha}$. Since the metrics $d^1$ and $\bar{d}^1$ are equivalent, it is enough to prove that the topologies induced by the metrics $\bar{d}^1$ and $d^0$ are the same. Since $\bar{d}^1(\alpha,\beta) \geq d^0(\alpha,\beta)$ for any $\alpha,\beta\in\mathcal{P}_{\kappa_1}^{\kappa_2}(\mat{P},\mat{Q})$, the topology induced by $\bar{d}^1$ is finer than the topology induced by $d^0$. So it suffices to prove the reciprocal.  

Given a sequence $\{\alpha_k\}_{k\in\mathbb{N}}$ which is convergent in $d^0$ to $\alpha_0$ we shall prove that it is also convergent in $\bar{d}^1$.

Suppose, by contrary, that $\bar{d}^1(\alpha_k,\alpha_0)\nrightarrow 0$ as $k\rightarrow \infty$.  Then there exists  some $\epsilon>0$, by taking a subsequence, still denoted by $\alpha_k$,  it holds that
$$\bar{d}^1(\alpha_k,\alpha_0)\geq \epsilon > 0.$$
Since $\displaystyle d^0(\alpha_k,\alpha_0) \to 0$, there exist a $k_0>0$ such that,  for $k\geq k_0$,
\begin{equation*}
\max_{t\in [0,1]}\{d(\dot{\alpha}_k(t),\dot{\alpha}_0(t))\}\geq \frac{3\epsilon}{4}.
\end{equation*}
 By Lemma \ref{lemconvlength}, $|\dot{\alpha}_{k}(t)|= L_{\alpha_k}\leq M$.
 
For every curve $\alpha_k\in \mathcal{P}_{\kappa_1}^{\kappa_2}(\mat{P},\mat{Q})$, $\dot{\alpha}_k$ is Lipschitz continuous and so 
 \begin{align}\label{eqlip5}
\dot{\alpha}_k(t_2)-\dot{\alpha}_k(t_1)& =\int_{t_1}^{t_2}\ddot{\alpha}_k(\tau)d\tau=L^2_{\alpha_k}\int_{t_1}^{t_2}\vet{t}'_{\alpha_k}(\tau)d\tau.
\end{align}
where $\dot{\alpha}_k(t)=L_{\alpha_k}\vet{t}_{\alpha_k}$, $\ddot{\alpha}_k(t)=L^2_{\alpha_k}\vet{t}_{\alpha_k}'$, for a.e. $t$. By Lemma \ref{lemconvlength} and \eqref{eqlip5}, the family $\{\dot{\alpha}_k\}_{k\in\mathbb{N}}$ has uniform bound and has the uniform Lipschitz constant.  This implies that the family $\{\dot{\alpha}_k\}_{k\in\mathbb{N}}$ is  equicontinuous.   By Azel\'a-Ascoli theorem, there exists a subsequence $\{\dot{\alpha}_{k_j}\}_{k\in\mathbb{N}}$ converges uniformly to a limit $v(t)$  which is a vector function defined at $[0,1]$. 
 
\noindent\textbf{Claim:}  $\dot{\alpha}_0=v(t)$.
 
 Since $\dot{\alpha}_0$ and $v(t)$ are continuous, it is sufficient to prove that for any $0\leq t_1\leq t_2\leq 1$, $\int_{t_1}^{t_2}(\dot{\alpha}_0(t)-v(t))dt=0$.
\begin{align}\label{tangentvector}
\int_{t_1}^{t_2}v(t)dt&=\int_{t_1}^{t_2}\lim_{j\rightarrow \infty}\dot{\alpha}_{k_j}(t)dt\nonumber\\
&=\lim_{j\rightarrow \infty}\int_{t_1}^{t_2}\dot{\alpha}_{k_j}(t)dt\\
&=\lim_{j\rightarrow \infty}\left(\alpha_{k_j}(t_2)-\alpha_{k_j}(t_1)\right)\nonumber\\
&=\alpha_0(t_2)-\alpha_0(t_1)\nonumber\\
&=\int_{t_1}^{t_2}\dot{\alpha}_0(t)dt.\nonumber
\end{align}
In the second equality of (\ref{tangentvector}), we used the uniform convergence of $\{\dot{\alpha}_{k_j}\}_{j\in\mathbb{N}}$ and the dominated convergence theorem. This proves the claim.
 
 By Claim, we have, for sufficiently large $j$,
 \begin{equation*}
\max_{t\in [0,1]}\left\{d(\dot{\alpha}_{k_j}(t),\dot{\alpha}_0(t))\right\}<\frac{3\epsilon}{4}.
\end{equation*}
Thus the contradiction happens. We have proved $\bar{d}^1(\alpha_k,\alpha_0)\rightarrow 0$ as $k\rightarrow\infty$.
\end{proof}

\section{Banach manifold  $\mathcal{P}_{\kappa_1}^{\kappa_2}(\mat{P},\mat{Q})$}\label{banach}

We observe that the space  $\mathcal{P}_{\kappa_1}^{\kappa_2}(\mat{P},\mat{Q})$  with $C^0$ (equivalently $C^1$) topology in Section \ref{C-topology} is not complete. In this section, we furnish $\mathcal{P}_{\kappa_1}^{\kappa_2}(\mat{P},\mat{Q})$ a complete norm such that it is a Banach manifold. In \cite{salzuh}, the  Saldanha and Zühlke  constructed a Hilbert manifold structure on a special subspace of the space of  so-called  $(\kappa_1, \kappa_2)$-admissible curves. We will use an approach similar to theirs.

Denote the space
 $\mathbf{E}_{\infty}=L^\infty[0,1]\times L^\infty[0,1]$, where $L^\infty[0,1]$ denotes the Banach space of essentially  bounded measurable functions on $[0,1]$ together with the essential supremum norm. Let $W^{1,\infty}[0,1]$ denote the Sobolev space of the functions in $L^\infty[0,1]$ with their weak derivatives also in $L^\infty[0,1]$. It is well known that a function in $W^{1,\infty}[0,1]$ is Lipschitz continuous.



If $\gamma(t): [0,1]\rightarrow \mathbb{S}^2$ is a smooth regular parameterized curve, its Frenet frame $\mathfrak{F}(t)$ satisfies 
$$\mathfrak{F}'_{\gamma}(t)=\mathfrak{F}_{\gamma}(t)\Lambda(t),$$
where
\begin{equation*}
\Lambda(t)=\left(\begin{array}{ccc} 0 & -|\dot{\gamma}(t)| & 0 \\
|\dot{\gamma}(t)| & 0 & -|\dot{\gamma}(t)|\kappa(t) \\
0 & |\dot{\gamma}(t)|\kappa(t) & 0 
\end{array}\right)\in \mathfrak{so}_3(\mathbb{R}). 
\end{equation*}

Now, given a  $\mat{P}\in\SO$ and a map $\Lambda:[0,1]\to \mathfrak{so}_3(\mathbb{R})$ of the form:
\begin{equation*}
\Lambda(t)=\left(\begin{array}{ccc} 0 & -v(t) & 0 \\
v(t) & 0 & -w(t) \\
0 & w(t) & 0 
\end{array}\right),
\end{equation*}

\noindent where $v,w\in L^{\infty}[0,1]$ and $v(t) >0 $. By the ODE theory (see \cite{or}, Theorem 3.4),  the initial value problem:
\begin{equation}\label{ivp}
\dot{\Phi}(t) = \Phi(t)\Lambda(t),\quad \Phi(0)= \mat{P},
\end{equation}
exists the unique solution $\Phi: [0,1]\to \SO$ which is of $W^{1,\infty}[0,1]$.
Let $h:(0,+\infty)\to\real $ be the smooth diffeomorphism
$$ h(t) = t - t^{-1} .$$
For each real number pair $\kappa_1<\kappa_2$, let $h_{\kappa_1,\kappa_2}:(\kappa_1,\kappa_2)\to\real$ be the smooth diffeomorphism
$$h_{\kappa_1,\kappa_2}(t) = (\kappa_1-t)^{-1} + (\kappa_2 - t)^{-1}. $$

If we take a pair  $(\hat{v},\hat{w})\in\mathbf{E}_{\infty}$, and 
$(v,w)$ given by
\begin{equation}\label{vw}
v(t)=h^{-1}(\hat{v}(t)),\quad w(t)=v(t)h^{-1}_{\kappa_1,\kappa_2}(\hat{w}(t)).
\end{equation}
By the definition of $h$ and $h_{\kappa_1,\kappa_2}$, it is straightforward to verify that $(v,w)\in \bf{E}_\infty$, $v(t)>0$ and $w(t)\in (\kappa_1, \kappa_2)$, $t\in [0,1]$. Hence we have the following definition:
\begin{definition} A parameterized curve $\gamma:[0,1]\to\mathbb{S}^2$ is called \emph{$(\kappa_1,\kappa_2)$-strongly admissible} if there exist $\mat{P}\in\SO$ and a pair $(\hat{v},\hat{w})\in\mathbf{E}_{\infty}$ such that $\gamma(t)=\Phi(t)(1,0,0)$ for all $t\in [0,1]$, where $\Phi$ is the unique solution in $W^{1,\infty}[0,1]$ to the initial value problem \eqref{ivp}, with $v,w$ given by Equation (\ref{vw}).


 Let $\mathcal{R}_{\kappa_1}^{\kappa_2}(\mat{P},\cdot)$ denote the set of all $(\kappa_1,\kappa_2)$-strongly admissible  parameterized curves $\gamma$ such that $\Phi_\gamma(0) = \mat{P}$.
 
We also define $\mathcal{R}_{\kappa_1}^{\kappa_2}(\mat{P},\mat{Q})$ to be the subspace of $\mathcal{R}_{\kappa_1}^{\kappa_2}(\mat{P},\cdot)$ satisfying
$$\Phi_\gamma(0) = \mat{P} \quad \text{and} \quad \Phi_\gamma(1)=\mat{Q},$$
where $\Phi_\gamma$ is the Frenet frame of $\gamma$. 

The set $\mathcal{R}_{\kappa_1}^{\kappa_2}(\mat{P},\cdot)$ is identified with $\mathbf{E}_\infty$ via correspondence $\gamma\leftrightarrow (\hat{v},\hat{w})$. 

\end{definition}

The above  identification   induces a norm $\|\cdot\|_B$ in $\mathcal{R}_{\kappa_1}^{\kappa_2}(\mat{P},\cdot)$ such that it becomes a Banach space and hence a (trivial) Banach manifold.
Moreover, since $\SO$ has dimension $3$, $\mathcal{R}_{\kappa_1}^{\kappa_2}(\mat{P},\mat{Q})$ is a closed subspace of codimension $3$ in $\mathcal{R}_{\kappa_1}^{\kappa_2}(\mat{P},\cdot)$ and a Banach space. 

Now we prove that


\begin{proposition}\label{prop13} Let $-\infty<\kappa_1<\kappa_2< +\infty$ and $\mat{P},\mat{Q}\in\SO$. Any parameterized curve $\gamma \in \mathcal{R}_{\kappa_1}^{\kappa_2}(\mP,\mQ)$ can be reparameterized by arc-length $s$, such that it becomes a $C^1$ regular parameterized curve and $[\gamma(s)]\in \mathcal{P}_{\kappa_1}^{\kappa_2}(\mat{P},\mat{Q})$.
\end{proposition}

\begin{proof} 
The proof is divided by several steps. 
Let $\gamma: [0,1]\to\mathbb{S}^2$ be a curve in $\mathcal{R}_{\kappa_1}^{\kappa_2}(\mat{P},\mat{Q})$.
 
(i). we prove the existence of the arc-length $s$ of $\gamma$ and show some of properties of $\gamma(s)$.  

By the definition of $\mathcal{R}_{\kappa_1}^{\kappa_2}(\mat{P},\mat{Q})$,   the frame $\Phi_\gamma$ of $\gamma$ is of $W^{1,\infty}[0,1]$. As a component of $\Phi_\gamma$, the function $\gamma$ is Lipschitz continuous. This implies that $\dot{\gamma}(t)$ exists a.e. $t$ and $|\dot{\gamma}(t)|$ is bounded.  Further  the arc-length of $\gamma$, as a curve, is well defined and is equal to
$$s(t)=\int_0^t|\dot{\gamma}(t)|dt,\quad t\in [0,1].$$ 
where $s(t)$ is Lipschitz continuous.
By the equation $\dot{\gamma}(t)=v(t)\vet{t}(t)$ a.e.,
$$s(t)=\int_0^t|\dot{\gamma}(t)|dt=\int_a^tv(t)dt,\quad t\in [0,1].$$

Note that  $v(t)>0$ in $[0,1] $ for a.e. $t\in[0,1]$. $s(t)$ is strictly increasing in $[0,1]$ and exists a strictly increasing continuous inverse function $t(s), s\in [0, L_{\gamma}]$. Hence, $t'(s)$ exists for a.e. $s\in [0, L_{\gamma}]$, where $t'(s)=\frac{1}{v(t(s))}$. In addition, observe that $v(t)=h^{-1}(\hat{v}(t))$ and $\hat{v}(t)\in L^\infty[0,1]$. Then $v(t)=\frac{\hat{v}+\sqrt{\hat{v}^2+4}}{2}$ is bounded below by a positive number for a.e. $t\in [0,1]$. This implies that  $v(t(s))$ is bounded below by a positive number for a.e. $s\in [0, L_{\gamma}]$. in $L^\infty[0,L_\gamma]$ and hence $t(s)$ is Lipschitz continuous. 

 Now we reparameterize $\gamma$ by arc-length $s$, that is, $\gamma(s)=\gamma(t(s)), s\in [0, L_{\gamma}]$. 

It is known  that Lipschitz continuity of $\gamma(t)$ implies Lipschitz continuity of $\gamma(s)$ (see \cite{auscher} Theorem 3.2), that is $\gamma(s)\in W^{1,\infty}[0,L_{\gamma}]$. Moreover, $\gamma'(s)$ exists for a.e. $s$ and $|\gamma'(s)|=1$ for a.e. $s$ (see \cite{auscher} Corollary 3.7). From the facts on $\gamma(t)$ and $\gamma(s)$ we have obtained above, it holds that
$$\gamma'(s)=\dot{\gamma}(t(s))\cdot t'(s)=\frac{\dot{\gamma}}{v}(t(s)) \quad \text{for a.e. $s$}.$$

By the differential system \eqref{ivp},  we have that, for a.e. $s$, 
$$\gamma'(s)=\frac{\dot{\gamma}}{v}(t(s))=\vet{t}(s),$$  
where $\displaystyle \vet{t}(s)=\vet{t}(t(s))).$

(ii). We will confirm that $\gamma'(s)$ exists for all $s$, $\gamma'(s)=\vet{t}(s)$ and $\gamma(s)$ is a $C^1$ regular curve.

In fact,  for any $s$, since $\gamma(s)$ is Lipschitz continuous and hence absolutely continuous, 
\begin{align}\label{derivativearc}
\gamma'(s)&=\lim_{\Delta s\rightarrow 0}\frac{\gamma(s+\Delta s)-\gamma(s)}{\Delta s}
=\lim_{\Delta s\rightarrow 0}\frac{\int_s^{s+\Delta s}\gamma'(u)du}{\Delta s}\nonumber\\
&=\lim_{\Delta s\rightarrow 0}\frac{\int_s^{s+\Delta s}{\vet{t}(u)}du}{\Delta s}.
\end{align} 
Note that $\vet{t}(t)$, as a component of  $\Phi_\gamma$, is in $W^{1, \infty}[0,1]$, that is, $\vet{t}(t)$ is Lipschitz continuous. So $\vet{t}(s)=\vet{t}(t(s))$ is continuous for $s$. By the mean value theorem, \eqref{derivativearc} implies that
\begin{align}\label{derivativearclength}
\gamma'(s)&
={\vet{t}(s)}, \quad s\in [0,1].
\end{align} 
 So we have proved that  $\gamma(s)$ is a $C^1$ curve. We mention that one may reparameterize $\gamma$ with a new parameter, still denoted by $t$,  such that $\gamma(t)$, $t\in [0,1]$ has the constant speed $L_{\gamma}$ and  is in $C^1[0,1]$.
 
(iii). We will confirm $[\gamma(s)]\in\mathcal{P}_{\kappa_1}^{\kappa_2}(\mat{P},\mat{Q})$.  

Since the parameter $t(s)$ is Lipschitz continuous, $\vet{t}(s)=\vet{t}(t(s))$ is  Lipschitz continuous, that is, it is of $W^{1,\infty}[0,L_{\gamma}]$. By \eqref{ivp}, the following equations hold for a.e. $s$,
\begin{align}\label{derivativetangent}\vet{t}'(s)=\left(-\gamma +\frac{w}{v}\vet{n}\right)(s),
\end{align}
where $ \displaystyle  \vet{n}(s)=\vet{n}(t(s))$ and $\frac{w(t)}{v(t)}\in (\kappa_1,\kappa_2).$

Note that $\hat{w}(t)\in L^{\infty}[0,1]$ and $w(t)=v(t)h^{-1}_{\kappa_1,\kappa_2}(\hat{w}(t))$.  It can be implied by the definition of $h_{\kappa_1, \kappa_2}$ that $$\kappa_1<\essinf_{t\in[0,1]}\frac{w(t)}{v(t)}\leq \esssup_{t\in[0,1]}\frac{w(t)}{v(t)}<\kappa_2.$$
Taking $\gamma(s)=\gamma(t(s))$, we have proved $[\gamma(s)]\in\mathcal{P}_{\kappa_1}^{\kappa_2}(\mat{P},\mat{Q})$. 

\end{proof}


For convenience, we give the following notation
\begin{definition} 
Let $\mathfrak{L}_{\kappa_1}^{\kappa_2}(\mat{P},\mQ)$ be the subset of $\mathcal{R}_{\kappa_1}^{\kappa_2}(\mat{P},\mQ)$ satisfying $\hat{v}(t)\equiv \hat{v}\in\mathbb{R}$.
\end{definition}
Now we are ready to prove Theorem \ref{thm3}, that is ,
 \begin{theorem}\label{thm-banach} $\mathcal{P}_{\kappa_1}^{\kappa_2}(\mat{P},\mat{Q})$ can be furnished a complete norm so that it is  a Banach space, hence a trivial Banach manifold.
\end{theorem}
\begin{proof} First, we confirm the  claim: $\mathfrak{L}_{\kappa_1}^{\kappa_2}(\mat{P},\mat{Q})=\mathcal{P}_{\kappa_1}^{\kappa_2}(\mat{P},\mat{Q})$.

Note that $\hat{v}(t)\equiv \hat{v}$ is constant.  Proposition \ref{prop13} implies that a parameterized curve $\gamma(t), t\in [0,1],$ in $\mathfrak{L}_{\kappa_1}^{\kappa_2}(\mat{P},\mQ)$ is  $C^1$ regular with its parameter  $t$ proportional to the arc-length and $[\gamma(t)]\in \mathcal{P}_{\kappa_1}^{\kappa_2}(\mat{P},\mat{Q})$. Reciprocally,
let $[\gamma]\in \mathcal{P}_{\kappa_1}^{\kappa_2}(\mat{P},\mat{Q})$, where $\gamma(t)$ is  taken to be  a $C^1$ regular parameterized curve with parameter $t\in [0,1]$ proportional to the arc-length.  Then it  is directly  verified that $\gamma(t)$ satisfies the ODE system (\ref{ivp}) for a.e. $t$,
\begin{equation*}
\dot{\Phi}(t) = \Phi(t)\Lambda(t)\quad \text{and}\quad \Phi(0)= \mat{P},
\end{equation*}
with $v(t)=|\dot{\gamma}(t)|\equiv L_\gamma$, $w(t)=L_{\gamma}\kappa(t).$

Take $\hat{v}(t)=h(L_{\gamma})$ and $\hat{w}(t)=h_{\kappa_1,\kappa_2}(\kappa(t))$. Then $\hat{v}\in\mathbb{R}$ and $\hat{w}\in L^{\infty}[0,1]$. Thus $\gamma(t)\in \mathfrak{L}_{\kappa_1}^{\kappa_2}(\mat{P},\mat{Q})$. 
So the claim holds. 

Note that $\mathfrak{L}_{\kappa_1}^{\kappa_2}(\mat{P},\mQ)$ is a closed subspace of $\mathcal{R}_{\kappa_1}^{\kappa_2}(\mat{P},\mQ)$. In fact, it can be obtained  via correspondence $\gamma\leftrightarrow (v,\hat{w})\in \mathbb{R}\times L^\infty[0,1] $ and $\Phi_\gamma(1)=\mat{Q}$. Hence $\mathfrak{L}_{\kappa_1}^{\kappa_2}(\mat{P},\mQ)$  has an induced complete norm. By the claim, 
$\mathcal{P}_{\kappa_1}^{\kappa_2}(\mat{P},\mat{Q})$ is a Banach space.

\end{proof}



\begin{thebibliography}{9}

\bibitem{auscher} 
Pascal Auscher, T. Coulhon and Alexander Grigoryan.
\textit{Heat Kernels and Analysis on Manifolds, Graphs, and Metric Spaces: Lecture Notes from a Quarter Program on Heat Kernels, Random Walks, and Analysis on Manifolds and Graphs: April 16 -July 13, 2002, Emile Borel Centre of the Henri Poincar\'e Institute, Paris, France }. 
American Mathematical Soc., 2003 - Mathematics - 423 pages

\bibitem{hatc}
A. Hatcher,
\textit{Algebraic Topology}.
Cambridge University Press, 2002.

\bibitem{shakhe}
B. A. Khesin and B. Z. Shapiro,
\textit{Homotopy classification of nondegenerate quasiperiodic curves on the 2-sphere}.
Publicatons de L'Institut Mathématique, Nouvelle série, tome 66 (80), 1999, 127-156.

\bibitem{little} 
J. A. Little,
\textit{Nondegenerate homotopies of curves on the unit 2-sphere}.
J. Differential Geometry, 4 (1970), 339-348.
\bibitem{or}
D. 0 'Regan,
\textit{Existence Theory for
Nonlinear Ordinary
Differential Equations}.
Mathematics and Its Applications, v. 398, 1997.

\bibitem{sald1}
N. C. Saldanha,
\textit{The cohomology of spaces of locally convex curves in the sphere -- I}.
arXiv:0905.2111.

\bibitem{sald2}
N. C. Saldanha,
\textit{The cohomology of spaces of locally convex curves in the sphere -- II}.
arXiv:0905.2116.
 
\bibitem{sald} 
N. C. Saldanha,
\textit{The homotopy type of spaces of locally convex curves in the sphere}.
Geom. Topol. 19 (2015), 1155-1203.

\bibitem{salsha}
N. C. Saldanha and B. Z. Shapiro,
\textit{Spaces of locally convex curves in $S^n$ and combinatorics of the group $B_{n+1}^+$}.
Journal of Singularities 4 (2012), 1-22.

\bibitem{shasha}
B. Z. Shapiro and M. Z. Shapiro,
\textit{On the number of connected components in the space of closed nondegenerate curves on $S^n$}.
Bulletin of the AMS 25 (1991), no. 1, 75-79.
\bibitem{salzuh}
N. C. Saldanha and P. Zühlke,
\textit{On the components of spaces of curves on the 2-sphere with geodesic curvature in a prescribed interval}.
Int. J. Math, vol. 24 (2013), no. 14, 1-78.
\bibitem{smale}
S. Smale,
\textit{Regular curves on Riemannian manifolds}.
Trans. Amer. Math. Soc. 87 (1956), no. 2, 492-512.

\bibitem{shap}
M. Z. Shapiro,
\textit{Topology of the space of nondegenerate curves}.
Math. USSR 57 (1993), 106-126.
\bibitem{zhou}
C. Zhou
\textit{On the homology of the space of curves immersed in the sphere with curvature constrained to a prescribed interval}.
arXiv 1809.05612.

\end{thebibliography}
\end{document}